\documentclass{article}

\usepackage{arxiv}

\usepackage[utf8]{inputenc} % allow utf-8 input
\usepackage[T1]{fontenc}    % use 8-bit T1 fonts
\usepackage{hyperref}       % hyperlinks
\usepackage{url}            % simple URL typesetting
\usepackage{booktabs}       % professional-quality tables
\usepackage{amsfonts}       % blackboard math symbols
\usepackage{nicefrac}       % compact symbols for 1/2, etc.
\usepackage{microtype}      % microtypography
\usepackage{lipsum}		% Can be removed after putting your text content
\usepackage{graphicx}
\usepackage{doi}

\usepackage{amsmath}
\usepackage{amssymb}
\usepackage{amsthm} %for theorems
\usepackage{gensymb}
\usepackage{mathtools}
\usepackage{cancel} % for 'cancel to' arrows
\usepackage{nicefrac}
\usepackage{upgreek} %for upright greek letters
\usepackage{enumitem} %change enum style ie. \begin{enumerate}[label=(\alph*)]
\usepackage{bookmark}
\usepackage{listings} %for including code with syntax highlighting
\usepackage{epstopdf} %allows eps files to be included in figures
\usepackage{mathrsfs} %for formal script font
\usepackage{mathdots} %for more dots
\usepackage{longtable} %for long tables that can span multiple pages
\usepackage[toc,page]{appendix}
\usepackage{color}
\usepackage{multibib}

\usepackage{algorithm}
\usepackage[noend]{algpseudocode} % for writing algorithms

\usepackage{subcaption} %for subfigures
\usepackage{mwe}

\usepackage{booktabs}

\title{Ordinal Optimisation for the Gaussian Copula Model}

\author{ Robert Chin\thanks{Department of Electrical and Electronic Engineering, The University of Melbourne, Australia \& School of Computer Science, University of Birmingham, United Kingdom. Email: \texttt{chinr@student.unimelb.edu.au}}
	\And
	Jonathan E. Rowe\thanks{School of Computer Science, University of Birmingham, United Kingdom \& The Alan Turing Institute, United Kingdom}
	\And
	Iman Shames\thanks{College of Engineering \& Computer Science, Australian National University, Australia} \\
	\And
	Chris Manzie\thanks{Department of Electrical and Electronic Engineering, The University of Melbourne, Australia} \\
	\And
	Dragan Ne\v{s}i\'{c}\thanks{Department of Electrical and Electronic Engineering, The University of Melbourne, Australia}
}

\date{}

\renewcommand{\headeright}{Technical Report}
\renewcommand{\undertitle}{Technical Report}
\numberwithin{equation}{section} %Label equations based on sections

\makeatletter
\newcommand{\pushright}[1]{\ifmeasuring@#1\else\omit\hfill$\displaystyle#1$\fi\ignorespaces}
\newcommand{\pushleft}[1]{\ifmeasuring@#1\else\omit$\displaystyle#1$\hfill\fi\ignorespaces}
\makeatother

\newtheorem{problem}{Problem}[section]
\newtheorem{definition}{Definition}[section]
\newtheorem{proposition}{Proposition}[section]
\newtheorem{remark}{Remark}[section]
\newtheorem{theorem}{Theorem}[section]
\newtheorem{lemma}{Lemma}[section]
\newtheorem{corollary}{Corollary}[section]

\begin{document}
\maketitle

\begin{abstract}
We present results on the estimation and evaluation of success probabilities for ordinal optimisation over uncountable sets (such as subsets of $\mathbb{R}^{d}$). Our formulation invokes an assumption of a Gaussian copula model, and we show that the success probability can be equivalently computed by assuming a special case of additive noise. We formally prove a lower bound on the success probability under the Gaussian copula model, and numerical experiments demonstrate that the lower bound yields a reasonable approximation to the actual success probability. Lastly, we showcase the utility of our results by guaranteeing high success probabilities with ordinal optimisation.
\end{abstract}

% keywords can be removed

\section{Introduction}

Ordinal optimisation (OO) is an approach introduced in \cite{Ho1992} for softening difficult problems in stochastic search and optimisation \cite{Spall2003}, and offered as a complementary approach to conventional optimisation techniques when there is `little hope' of finding the global optimum solution. The outcome provided by OO is a high-probability guarantee that one or more out of a selected subset of candidate solutions is an acceptable sub-optimal solution, and its operation rests on two underlying principles: 1) by selecting the subset according to order, the selection is more `robust' to noise; and 2) by `goal softening' (i.e. increasing the degree of sub-optimality), chances of success can be improved.

OO was primarily introduced to the control theory community for the simulation-based optimisation of discrete-event dynamic systems \cite{Ho1992}, and has seen numerous successful applications in design/search problems across different disciplines. It was applied to stochastic optimal control in \cite{Deng1999}, where OO was used to find a heuristic solution to the Witsenhausen problem \cite{Witsenhausen1968}. A proposed solution 50\% better than Witsenhausen's own proposed solution was found in terms of performance cost. In \cite{Ho1995}, OO was applied to rare event simulation of overflow probabilities in queuing systems. By embracing goal softening, computational requirements for simulation were reduced by approximately 3 orders of magnitude. An improvement by a factor of 100 was also reported by \cite{Zhang2015} in the time taken to generate optimal cloud computing schedules by an OO method, compared to a Monte-Carlo approach.

A research area with roots from OO is the optimal computing budget allocation (OCBA) framework, which addresses the problem of efficient allocation of simulation resources \cite{Chen2000}. In the literature, the OCBA framework has also been referred to as ranking and selection \cite{Peng2018}, as well as `ordinal optimisation' \cite{Glynn2004}. The early results from the OCBA framework may also be regarded as a precursor of a pure exploration objective from the vast multi-armed bandit literature \cite{Lattimore2020}. In this paper, we refer to OO as the traditional, non-sequential setting as found in \cite[Chapter II]{Ho2007}, in which the `horse-race' selection rule is provably optimal \cite{Yang2002}. To date, OO has been formulated as a search problem over a finite search space. That is, the variables which encode all the possible solutions take on values from a finite set. There are of course problems where the search space will be infinite and possibly also uncountable (e.g. optimisation over the cone of positive semi-definite matrices for controller tuning \cite{Chin2019}, \cite{Maass2020}). For these problems, OO can still be informally applied in practice by conditioning on the expected number of acceptable solutions in the sample, accompanied with the assurance that a large enough sample from the search space becomes `representative' of the search space itself \cite{Lin2002}, and thus can be used as a good heuristic. In this paper, we present a reformulation of the original OO problem which is motivated by formally extending OO to search over uncountable sets.

Our contributions provide theoretical results under particular specialisations to our ordinal optimisation formulation, namely an additive noise and Gaussian copula model introduced in Section \ref{sec:prob}. Copula models see a host of applications in engineering, e.g. \cite{Zeng2014} for simulating communications channels, and finance, e.g. \cite{Cheng2007} for modelling returns. The usefulness of copulae stem from Sklar's theorem, which informally, says that any continuous multivariate distribution can be equivalently represented by its marginal distributions and a copula \cite[Theorem 2.3.3]{Nelsen2006}. In Section \ref{sec:prelim}, we focus on the additive noise model, and state results for both cases when the distributions are known and unknown. In Section \ref{sec:main}, we show that the success probability under a Gaussian copula model can be computed by assuming a special case of additive noise. We develop a lower bound for the success probability, and numerically compare the bound to other approximations. In Section \ref{sec:cases}, we demonstrate some of our results for guaranteeing high probabilities of success with ordinal optimisation.

\section{Preliminaries}
\label{sec:prob}

\subsection{Notation}
\label{sec:notation}
Throughout this paper, the symbols $\leq$ and $<$ refer to element-wise inequalities between vectors. The set $\mathbb{R}$ denotes the real numbers, $\mathbb{R}_{\geq 0}$ denotes the non-negative reals, and $\mathbb{N}$ denotes the set of natural numbers. Random vectors are written in bold capital (e.g. $\mathbf{X}$). Non-random vectors are denoted in bold lowercase (e.g. $\mathbf{v}$). The symbol $\top$ is used as superscript to indicate the matrix transpose. Following the notation of \cite{Shaked2007}, stochastic dominance is denoted by $\underset{\mathrm{st}}{\preceq}$ and the symbol $\underset{\mathrm{st}}{=}$ between random elements denotes equality in law. The standard Gaussian probability density function (PDF), cumulative distribution function (CDF), inverse CDF (i.e. quantile function) and $Q$-function (complementary CDF) are canonically represented using $\phi\left(\cdot\right)$, $\Phi\left(\cdot\right)$, $\Phi^{-1}\left(\cdot\right)$, $Q\left(\cdot\right)$ respectively. We write $\mathbf{X} \sim \mathcal{N}\left(\boldsymbol{\mu}, C\right)$ to denote that $\mathbf{X}$ is Gaussian distributed with mean $\boldsymbol{\mu}$ and covariance $C$. The probability of an event is measured by $\operatorname{Pr}\left(\cdot\right)$ with respect to a probability space that is clear from context. The abbreviation i.i.d. stands for mutually independent and identically distributed. The symbol $\mathbb{E}$ denotes mathematical expectation, the variance of a random variable $X$ is written as $\operatorname{Var}\left(X\right)$ and the covariance between $X$ and $Y$ is $\operatorname{Cov}\left(X, Y\right)$. The notation $\left[\mathbf{X} \middle| \mathbf{Y} = \mathbf{y}\right]$ is understood to mean a random vector that is equal in law to the conditional distribution of $\mathbf{X}$ given $\mathbf{Y} = \mathbf{y}$. The $k$\textsuperscript{th} order statistic of a i.i.d. sample of size $n$ from parent distribution $Z$ will be denoted by $Z_{k:n}$. The binomial coefficient is denoted $\binom{n}{k}$. The symbols $\mathbf{1}$ and $I$ denote a vector of ones and the identity matrix respectively (with dimensions clear from context). The logarithm $\log\left(\cdot\right)$ is taken to mean the natural logarithm, while $\cot\left(\cdot\right)$ is the cotangent function.

\subsection{Formulation of Ordinal Optimisation}
\label{sec:formulation}

We introduce the following class of problem.

\begin{problem}[Ordinal optimisation]
Consider $n$ i.i.d. copies of $\left(Z_{i}, X_{i}\right)$ drawn from some bivariate distribution of $\left(Z, X\right)$. We observe $Z_{1}, \dots, Z_{n}$, and order these observations from best to worst, denoted by $Z_{1:n} \leq \dots \leq Z_{n:n}$. The best $m$ are selected, given by $Z_{1:n} \leq \dots \leq Z_{m:n}$, with respective $X$-values denoted as $X_{\left\langle 1\right\rangle}, \dots, X_{\left\langle m\right\rangle}$, which are initially unobserved. More explicitly, we have selected the pairs $\left(Z_{1:n}, X_{\left\langle 1\right\rangle}\right), \dots, \left(Z_{m:n}, X_{\left\langle m\right\rangle}\right)$. The success probability is defined as
\begin{equation}
p_{\mathrm{success}}\left(n, m, \alpha\right) := \operatorname{Pr}\left(\bigcup_{i = 1}^{m}\left\{X_{\left\langle i\right\rangle} \leq x_{\alpha}^{*}\right\}\right) = \operatorname{Pr}\left( \min_{i \in \left\{1, \dots, m\right\}}\left\{X_{\left\langle i\right\rangle}\right\} \leq x_{\alpha}^{*}\right).
\label{eq:ord-opt-defn}
\end{equation}
where $x_{\alpha}^{*}$ is the $100\alpha$ percentile of the distribution of $X$. That is, what is the probability that at least one of the selected $m$ is within the best $100\alpha$ percentile?
\label{prob:ord_opt}
\end{problem}
If we distill the above problem into its key elements, the overall theme is that of uncertain optimal selection, whereby $X$ represents the true value trying to be optimised, while $Z$ represents the actual observations (which may be noisy versions of the $X$-values). There is also a notion of soft optimisation for the selected candidates, defined in this paper as probability of performance within best $100\alpha$ percentile, where $\alpha$ plays the role of goal softening (although our framework can be adapted to work with other types of goal softening, such as probability of performance below some nominal level). We suggest some applications where Problem \ref{prob:ord_opt} could be a suitable model, and having the ability to quantify the probability of success would be useful.
\begin{itemize}
\item In job hiring, the $Z_{i}$ are job applicants' evaluations in an initial screening process, and the $X_{i}$ are evaluations of hired or interviewed applicants.
\item In robotic control, the $Z_{i}$ are the performances of tuned/trained robotic controllers in an offline simulation, and the $X_{i}$ are tested performances of controllers on the physical robot.
\item In vaccine discovery, the $Z_{i}$ are the effectiveness of vaccine candidates in computational experiments, and the $X_{i}$ are the effectiveness of the vaccine candidates in clinical trials.
\end{itemize}

The original OO formulation in \cite[Chapter II]{Ho2007} is related to Problem \ref{prob:ord_opt} in the sense that the $\left(Z_{i}, X_{i}\right)$ are not i.i.d.; rather the $X_{i}$ are sampled without replacement from a finite population, and $Z_{i}$ is related to $X_{i}$ via additive noise. To arrive at Problem \ref{prob:ord_opt} under an additive noise model, we can take the support of $X$ to be a continuum, where the distributions of $X_{i}$ and $Z_{i}$ themselves are induced by an arbitrary/black-box i.i.d. sampling mechanism from an uncountably large population (thus transforming the problem into search over an uncountable set). Interestingly, if Problem \ref{prob:ord_opt} is then conditioned on the cardinality of $\left\{i:X_{\left\langle i\right\rangle}\leq x_{\alpha}^{*},1\leq i\leq n\right\}$, i.e. the number of ideal/acceptable candidates in the initial sample, the conditional success probability then coincides with the `alignment probability' in original OO. These alignment probabilities have typically been approximated via Monte-Carlo simulation \cite{Lau1997}. In this paper, we discuss other approaches for computing/estimating the unconditional success probabilities.

\subsection{Problem Statement}

In this subsection, we introduce two problem statements, which are subsumed by Problem \ref{prob:ord_opt}. The first is an additive noise model. Under this refinement, we denote the success probability as $p_{\mathrm{success}}^{\mathrm{A}}$, followed by dependencies in parentheses as required. The specialised problem is formally stated as follows.
\begin{problem}[Additive noise OO]
In addendum to Problem \ref{prob:ord_opt}, the causal mechanism for $Z_{i}$ is given by
\begin{equation}
Z_{i} = X_{i} + Y_{i}
\end{equation}
where $X_{i}$ and $Y_{i}$ are independent, and $Y_{i}$ may be viewed as a noise term. Moreover, both $X_{i}$, $Y_{i}$ are continuous random variables. Denote $f_{X}\left(x\right)$ and $F_{X}\left(x\right)$ as the PDF and CDF respectively of each $X_{i}$, and $f_{Y}\left(x\right)$ and $F_{Y}\left(y\right)$ as the PDF and CDF respectively of each $Y_{i}$. We seek to:
\begin{enumerate}[label=(\alph*)]
\item evaluate $p_{\mathrm{success}}^{\mathrm{A}}\left(n, m, \alpha\right)$ if the distributions of $X_{i}$ and $Y_{i}$ are known, \label{prob:additive-noise-part-a}
\item as a special case of \ref{prob:additive-noise-part-a}, evaluate $p_{\mathrm{success}}^{\mathrm{A}}\left(n, m, \alpha, \xi\right)$ when it is known in particular that $X_{i} \sim \mathcal{N}\left(0, 1\right)$ and $Y_{i} \sim \mathcal{N}\left(0, \xi^{2}\right)$, where we refer to $\xi^{2} = \operatorname{Var}\left(Y_{i}\right)/\operatorname{Var}\left(X_{i}\right)$ as the noise-to-signal ratio, \label{prob:additive-noise-part-b} 
\item provide bounds for $p_{\mathrm{success}}^{\mathrm{A}}\left(n, m, \alpha\right)$ if the distributions of $X_{i}$ and $Y_{i}$ are not known. \label{prob:additive-noise-part-c}
\end{enumerate}
\label{prob:additive-noise}
\end{problem}

Note that a solution to \ref{prob:additive-noise-part-a} of course implies a solution to \ref{prob:additive-noise-part-b}, however we list \ref{prob:additive-noise-part-b} separately due to its connection with the upcoming problem, which is the focus of the main results in this paper. The second model considered is a Gaussian copula for $\left(X_{i}, Z_{i}\right)$, which imposes a particular class of copulae on the joint distribution.
\begin{definition}[Gaussian copula model]
The Gaussian copula is a family of copulae, parametrised by a correlation matrix $\Sigma$, and given by the joint CDF
\begin{equation}
C\left(u_{1}, \dots, u_{d}; \Sigma\right) = \boldsymbol{\Phi}_{\mathbf{0}, \Sigma}\left(\Phi^{-1}\left(u_{1}\right), \dots, \Phi^{-1}\left(u_{d}\right)\right)
\end{equation}
where $\boldsymbol{\Phi}_{\mathbf{0}, \Sigma}$ is the CDF of the multivariate Gaussian with zero mean and covariance matrix equal to correlation matrix $\Sigma$.
\end{definition}

Under this refinement, we denote the success probability as $p_{\mathrm{success}}^{\mathrm{G}}$, followed by dependencies in parentheses as required. The specialised problem is formally stated as follows.

\begin{problem}[Gaussian copula OO]
In addendum to Problem \ref{prob:ord_opt}, the $\left(Z_{i}, X_{i}\right)$ are i.i.d. copies from a bivariate distribution with a Gaussian copula with correlation of $\rho > 0$. We seek to:
\begin{enumerate}[label=(\alph*)]
\item evaluate $p_{\mathrm{success}}^{\mathrm{G}}\left(n, m, \alpha, \rho\right)$ directly, \label{prob:copula-part-a}
\item bound $p_{\mathrm{success}}^{\mathrm{G}}\left(n, m, \alpha, \rho\right)$ with estimates. \label{prob:copula-part-b}
\end{enumerate}
\label{prob:copula}
\end{problem}

In the literature, the class of multivariate distributions with Gaussian copulae are known as `non-paranormal' distributions \cite{Liu2012a}, and alternatively as `meta-Gaussian' distributions \cite{Storvik2009}. Thus Problem \ref{prob:copula} asserts that $\left(Z, X\right)$ are non-paranormal/meta-Gaussian with positive correlation. Compared to the additive noise model, the Gaussian copula model allows for a generally non-additive noise representation. This is a justifiable model in many settings, as one only needs to look at the abundance of applications for the Gaussian copula in literature.

Although the Gaussian copula can model non-additive noise, it is later demonstrated in Theorem \ref{thm:special_case} that Problem \ref{prob:copula}\ref{prob:copula-part-a} can be addressed by considering a special case of the additive noise model by Problem \ref{prob:additive-noise}\ref{prob:additive-noise-part-b}. Hence we devote attention to both Problems \ref{prob:additive-noise} and \ref{prob:copula} in this paper.

\section{Ordinal Optimisation Under Additive Noise}
\label{sec:prelim}

In this section, we seek to establish a solution to Problem \ref{prob:additive-noise}. The distribution-free case \ref{prob:additive-noise-part-c} is considered first, followed by case \ref{prob:additive-noise-part-a} when the distribution of $X$ and $Y$ are known. 

\subsection{Distribution-Free Bounds}
\label{sec:dist-free}

\begin{proposition}[Distribution-free bounds]
The success probability from Problem \ref{prob:additive-noise} is bounded by
\begin{equation}
1 - \left(1 - \alpha\right)^{m} \leq p_{\mathrm{success}}^{\mathrm{A}}\left(n, m, \alpha\right) \leq 1 - \left(1 - \alpha\right)^{n}
\end{equation}
for any $n \in \mathbb{N}$, $m \in \left\{1, \dots, n\right\}$, $\alpha \in \left[0, 1\right]$.
\label{prop:dist-free-bound}
\end{proposition}
\begin{proof}
See Appendix \ref{sec:proof-dist-free-bound}.
\end{proof}
\begin{remark}
The lower bound depends only on $\alpha$ and $m$, while the upper bound only depends on $\alpha$ and $n$. The upper and lower bounds will be equal if $n = m$, or if $\alpha = 0$ or $\alpha = 1$.
\label{rem:n-equal-m}
\end{remark}
Note that the lower bound may be interpreted as the `blind-pick' success probability \cite{Ho2007}, which is the probability of success if the selected $m$ were uniformly randomly selected without replacement from the sample of $n$. The upper bound is also equal to the probability that there is any acceptable $X$-value in the sample at all (via the binomial distribution).

\subsection{Distribution-Known Success Probability}
\label{sec:p-success}

To address Problem \ref{prob:additive-noise}\ref{prob:additive-noise-part-a}, an exact formula for the success probability is available, if the full distribution functions of $X$ and $Y$ are known.
\begin{proposition}[Exact success probability]
The success probability from Problem \ref{prob:additive-noise} is given by the expression
\begin{equation}
p_{\mathrm{success}}^{\mathrm{A}}\left(n, m, \alpha\right)  = 1 - \left(1 - \alpha\right)^{n} - \sum_{g = 1}^{n - m}\binom{n}{g}\alpha^{g}\left(1 - \alpha\right)^{n - g}\left(\int_{-\infty}^{\infty}F_{Z_{\left\{ g+m\right\} }}\left(z\right)f_{Z_{\left\{ 1\right\} }}\left(z\right)dz\right)
\label{eq:p-success-expresion}
\end{equation}
where
\begin{gather}
F_{Z_{\left\{g+m\right\}}}\left(z\right) := \sum_{j = 0}^{n - g - m}\binom{n - g}{j}\left[1 - F_{\overline{Z}}\left(z\right)\right]^{j}\left[F_{\overline{Z}}\left(z\right)\right]^{n - g - j}, \label{eq:upper-conditional-cdf}\\
F_{\overline{Z}}\left(z\right) := \int_{x_{\alpha}^{*}}^{\infty}F_{Y}\left(z - x\right)f_{\overline{X}}\left(x\right)dx,  \label{eq:sum-dist-upper-cdf}\\
f_{\overline{X}}\left(x\right) := \begin{cases} \dfrac{f_{X}\left(x\right)}{1 - \alpha}, & x \geq x_{\alpha}^{*} \\
0, & x < x_{\alpha}^{*} \end{cases}, \label{eq:left-truncated-pdf} 
\end{gather}
and
\begin{gather}
f_{Z_{\left\{1\right\} }}\left(z\right) := g\left[1 - F_{\underline{Z}}\left(z\right)\right]^{g - 1}f_{\underline{Z}}\left(z\right), \label{eq:lower-conditional-pdf} \\
F_{\underline{Z}}\left(z\right) := \int_{-\infty}^{x_{\alpha}^{*}}F_{Y}\left(z - x\right)f_{\underline{X}}\left(x\right)dx,  \label{eq:sum-dist-lower-cdf} \\
f_{\underline{Z}}\left(z\right) := \int_{-\infty}^{x_{\alpha}^{*}}f_{Y}\left(z - x\right)f_{\underline{X}}\left(x\right)dx,  \label{eq:sum-dist-lower-pdf} \\
f_{\underline{X}}\left(x\right) := \begin{cases} \dfrac{f_{X}\left(x\right)}{\alpha}, & x \leq x_{\alpha}^{*} \\
0, & x > x_{\alpha}^{*} \end{cases}.  \label{eq:right-truncated-pdf}
\end{gather}
\label{prop:success-prob}
\end{proposition}
\begin{proof}
See Appendix \ref{sec:proof-success-prob}.
\end{proof}
While it may seem useful to have an exact expression, direct evaluation of $p_{\mathrm{success}}$ involves the nested evaluation of multiple integrals (which may not have any analytical form), hence even numerical evaluation of $p_{\mathrm{success}}$ can result in a `hefty' computation. Performing straightforward numerical integration using quadrature methods with Proposition \ref{prop:success-prob} is of time complexity $O\left(n\right)$, due to the main outer sum of \eqref{eq:p-success-expresion}. Binomial probabilities can be computed in $O\left(1\right)$ time using the log-gamma function \cite[Section 4.12]{Krishnamoorthy2016}, and we have assumed that numerical integration in one dimension can be performed in $O\left(1\right)$ time for a given error tolerance. Note that as the CDF of an order statistic, \eqref{eq:upper-conditional-cdf} can also be evaluated in $O\left(1\right)$ time by rewriting it as an integral involving the PDF, which resembles a binomial probability \cite[Eq. (2.2.2)]{Arnold2008}. Hence it can subsequently be evaluated in $O\left(1\right)$ using the aforementioned log-gamma approach.

However, the $O\left(n\right)$ time complexity can be misleading because the integral in \eqref{eq:p-success-expresion} involves multiple nested $O\left(1\right)$ computations. There will also be the approximation error associated with numerical integration.

\section{Ordinal Optimisation Under Gaussian Copula Model}
\label{sec:main}

This section contains our main results, as we study Problem \ref{prob:copula}. Case \ref{prob:copula-part-a} is first addressed via a connection to Problem \ref{prob:additive-noise}\ref{prob:additive-noise-part-b}. Then motivated by an approximation formula for $p_{\mathrm{success}}$, we address case \ref{prob:copula-part-b} using a stochastic dominance argument.

\subsection{Connection to Additive Noise}

The following result shows that under the Gaussian copula model, the $p_{\mathrm{success}}$ can be found via a special case of the additive noise problem.

\begin{theorem}[Connection between Gaussian copula model and additive noise]
Consider Problem \ref{prob:copula}\ref{prob:copula-part-a}. Then $p_{\mathrm{success}}^{\mathrm{G}}$ can be computed by instead solving Problem \ref{prob:additive-noise}\ref{prob:additive-noise-part-b}, with noise-to-signal ratio $\xi^{2} =1/\rho^{2} - 1$. Explicitly,
\begin{equation}
    p_{\mathrm{success}}^{\mathrm{G}}\left(n, m, \alpha, \rho\right) = p_{\mathrm{success}}^{\mathrm{A}}\left(n, m, \alpha, \dfrac{1}{\rho^{2}} - 1\right)
    \label{eq:special_case}
\end{equation}
\label{thm:special_case}
\end{theorem}
\begin{proof}
Firstly, note that a bivariate Gaussian distribution clearly has a Gaussian copula. Secondly, due to the specification from Problem \ref{prob:ord_opt}, performing separate univariate monotone transformations on each of $Z_{i}$ and $X_{i}$ will not affect $p_{\mathrm{success}}$, since the definition of the success probability only relies on ordinal selection of the $Z_{i}$, and the percentiles of $X_{i}$. Under a Gaussian copula model, we can assume without loss of generality that $\left(Z_{i}, X_{i}\right)$ are bivariate standard Gaussian with correlation $\rho$. Then we just need to show that additive noise Problem \ref{prob:additive-noise}\ref{prob:additive-noise-part-b} also admits a Gaussian copula representation with $\xi^{2} = 1/\rho^{2} - 1$. Since the sum $Z_{i} = X_{i} + Y_{i}$ will also be Gaussian, then $\left(Z_{i}, X_{i}\right)$ are bivariate Gaussian with correlation that can be computed by
\begin{align}
\rho &= \dfrac{\operatorname{Cov}\left(Z_{i}, X_{i}\right)}{\sqrt{\operatorname{Var}\left(Z_{i}\right)\operatorname{Var}\left(X_{i}\right)}} \\
&= \dfrac{\operatorname{Cov}\left(X_{i} + Y_{i}, X_{i}\right)}{\sqrt{1 + \xi^{2}}} \\
&=  \dfrac{1}{\sqrt{1 + \xi^{2}}}.
\end{align}
After inverting, we get $\xi^{2} = 1/\rho^{2} - 1$.
\end{proof}

Due to the equality in \eqref{eq:special_case}, Problem \ref{prob:copula} immediately enjoys the results from Section \ref{sec:prelim}. Under the Gaussian copula model, we do not explicitly require the marginal distributions $X_{i}$ and $Z_{i}$ to be known, but we can also assume without loss of generality that $X_{i} \sim \mathcal{N}\left(0, 1\right)$ and that $Z_{i}$ is obtained via independent additive Gaussian noise $Y_{i} \sim \mathcal{N}\left(0, \xi^{2}\right)$, hence we have the trivariate Gaussian representation
\begin{equation}
\begin{bmatrix}
X_{i} \\ Y_{i} \\ Z_{i}
\end{bmatrix} \sim \mathcal{N}\left(\begin{bmatrix}
0 \\ 0 \\ 0
\end{bmatrix}, \begin{bmatrix}
1 & 0 & 1 \\ 0 & \xi^{2} & \xi^{2} \\ 1 & \xi^{2} & 1 + \xi^{2}
\end{bmatrix}\right).
\label{eq:bivariate-gaussian}
\end{equation}
Recognise that the covariance matrix in \eqref{eq:bivariate-gaussian} is not full rank, primarily because of the relation $Z_{i} = X_{i} + Y_{i}$. A reduced-dimension equivalent representation (with respect to computing $p_{\mathrm{success}}$) is the bivariate Gaussian
\begin{equation}
\begin{bmatrix}
X_{i} \\ Z_{i}
\end{bmatrix} \sim \mathcal{N}\left(\begin{bmatrix}
0 \\ 0
\end{bmatrix}, \begin{bmatrix}
1 & \rho  \\ \rho & 1
\end{bmatrix}\right)
\label{eq:bivariate-gaussian-copula}
\end{equation}
that arises from the bivariate Gaussian copula model.

\subsection{Approximation Formula}
\label{sec:gaussian_approximation}

Motivated by the difficulty in evaluating $p_{\mathrm{success}}$ using Proposition \ref{prop:success-prob}, we develop a computationally tractable approximation for $p_{\mathrm{success}}$ under the Gaussian copula model. From Theorem \ref{thm:special_case} and the representation \eqref{eq:bivariate-gaussian-copula}, the success probability may be written as
\begin{equation}
p_{\mathrm{success}}^{\mathrm{G}} = \operatorname{Pr}\left(\min\left\{X_{\left\langle 1\right\rangle}, \dots, X_{\left\langle m\right\rangle}\right\} \leq \Phi^{-1}\left(\alpha\right)\right).
\end{equation}
Note the symmetry properties of the zero-mean Gaussian and its order statistics, i.e. we can relate the lower and upper extreme order statistics by
\begin{equation}
\left(Z_{1:n}, \dots, Z_{m:n}\right) \underset{\mathrm{st}}{=} \left(-Z_{n:n}, \dots, -Z_{\left(n - m + 1\right):n}\right).
\label{eq:symmetry_ord_stat}
\end{equation}
Then it follows that Problem \ref{prob:copula}\ref{prob:copula-part-b} is symmetric, in the sense
\begin{equation}
\operatorname{Pr}\left(\min\left\{X_{\left\langle 1\right\rangle}, \dots, X_{\left(m\right\rangle}\right\} \leq \Phi^{-1}\left(\alpha\right)\right) = \operatorname{Pr}\left(\max\left\{X_{\left\langle n - m + 1\right\rangle}, \dots, X_{\left\langle n\right\rangle}\right\} \geq \Phi^{-1}\left(1 - \alpha\right)\right).
\end{equation}
Hence, we develop the approximation of $p_{\mathrm{success}}$ using the upper variables $X_{\left\langle n - m + 1\right\rangle}$, $\dots$, $X_{\left\langle n\right\rangle}$ instead of the lower variables $X_{\left\langle 1\right\rangle}$, $\dots$, $X_{\left\langle m\right\rangle}$ (analogous to considering a maximisation problem rather than a minimisation problem). As will be clear, the resulting approximation involves the usual multivariate Gaussian CDF rather than the complementary CDF, for which software implementations of the former are more prevalent. Our approach is to approximate the joint extreme order statistics of the standard Gaussian $Z$ from representation \eqref{eq:bivariate-gaussian-copula} with a multivariate Gaussian. The distribution of the random vector $\mathbf{Z}' = \left(Z_{\left(n - m + 1\right):n}, \dots, Z_{n:n}\right)$ is chosen to be approximated with the distribution of
\begin{equation}
\widehat{\mathbf{Z}}' \sim \mathcal{N}\left(\boldsymbol{\mu}_{\widehat{\mathbf{Z}}'} , C_{\widehat{\mathbf{Z}}'} \right),
\label{eq:asymp_approx_distribution}
\end{equation}
with mean vector
\begin{equation}
\boldsymbol{\mu}_{\widehat{\mathbf{Z}}'} = \begin{bmatrix} \Phi^{-1}\left(p_{1}\right) &  \dots & \Phi^{-1}\left(p_{m}\right) \end{bmatrix}^{\top}
\label{eq:approx-mean-vector}
\end{equation}
and covariance structure
\begin{equation}
\left[C_{\widehat{\mathbf{Z}}'}\right]_{ij} = \dfrac{p_{i}\left(1-p_{j}\right)}{n\phi\left(\Phi^{-1}\left(p_{i}\right)\right)\phi\left(\Phi^{-1}\left(p_{j}\right)\right)}, \quad i \leq j,
\label{eq:approx-covariance-structure}
\end{equation}
where $p_{1} = \frac{n - m}{n}, \dots, p_{m} = \frac{n - 1}{n}$. This approximation is justified by the asymptotic normality of joint central order statistics \cite[Theorem 8.5.2]{Arnold2008}, with which we subsequently approximate the extreme order statistics. Note that the practice of approximating order statistics using asymptotic theory is not uncommon \cite{Reiss1989}. From \eqref{eq:bivariate-gaussian-copula}, the conditional distribution of $\mathbf{X}' = \left(X_{\left\langle n - m + 1\right\rangle}, \dots, X_{\left\langle n\right\rangle}\right)$ given $\mathbf{Z}'$ can be computed using well-known Gaussian conditioning formulae \cite[Eq. (A.6)]{Rasmussen2006} to be the Gaussian
\begin{equation}
\left[\mathbf{X}'\middle|\mathbf{Z}' = \mathbf{z}\right] \sim \mathcal{N}\left(\rho \mathbf{z}, \left(1 - \rho^{2}\right)I\right).
\label{eq:gaussian-conditioning-rho}
\end{equation}
With a Gaussian approximation for $\mathbf{Z}'$ and a Gaussian form for $\mathbf{X}'$ conditioned on $\mathbf{Z}'$, one can analytically marginalise out $\mathbf{Z}'$ with well-known formulae \cite[Eq. (1.11)]{Ristic2004} and obtain a Gaussian approximation for $\mathbf{X}'$ . Hence $\mathbf{X}'$ is approximated with a multivariate Gaussian vector $\widehat{\mathbf{X}}' \sim \mathcal{N}\left(\boldsymbol{\mu}_{\widehat{\mathbf{X}}'}, C_{\widehat{\mathbf{X}}'}\right)$, where
\begin{gather}
\boldsymbol{\mu}_{\widehat{\mathbf{X}}'} = \rho\boldsymbol{\mu}_{\widehat{\mathbf{Z}}'} \label{eq:marginalisation-mean} \\
C_{\widehat{\mathbf{X}}'} = \rho^{2}C_{\widehat{\mathbf{Z}}'} + \left(1 - \rho^{2}\right)I. \label{eq:marginalisation-covariance}
\end{gather}
We denote the CDF of this multivariate Gaussian by $\boldsymbol{\Phi}_{\boldsymbol{\mu}_{\widehat{\mathbf{X}}'}, C_{\widehat{\mathbf{X}}'}}\left(\cdot\right)$. The success probability is given by
\begin{align}
p_{\mathrm{success}}^{\mathrm{G}}\left(n, m, \alpha, \rho\right) &= \operatorname{Pr}\left(\max\left\{X_{\left\langle n - m + 1\right\rangle}, \dots, X_{\left\langle n\right\rangle}\right\} \geq \Phi^{-1}\left(1 - \alpha\right)\right) \\
&= 1 - \operatorname{Pr}\left(\mathbf{X}' \leq \Phi^{-1}\left(1 - \alpha\right)\mathbf{1}\right).
\end{align}
Using the Gaussian approximation $\widehat{\mathbf{X}}' \sim \mathcal{N}\left(\boldsymbol{\mu}_{\widehat{\mathbf{X}}'}, C_{\widehat{\mathbf{X}}'}\right)$ for $\mathbf{X}'$, the right-hand side of the above equation is approximated as
\begin{align}
\widehat{p}_{\mathrm{success}}\left(n, m, \alpha, \rho\right) &:= 1 - \operatorname{Pr}\left(\widehat{\mathbf{X}}' \leq \Phi^{-1}\left(1 - \alpha\right)\mathbf{1}\right) \\
&= 1 - \boldsymbol{\Phi}_{\boldsymbol{\mu}_{\widehat{\mathbf{X}}'}, C_{\widehat{\mathbf{X}}'}}\left(\Phi^{-1}\left(1 - \alpha\right)\mathbf{1}\right).
\label{eq:asymp_approx}
\end{align}
Computing this approximation formula is of time complexity $O\left(m^{2}\right)$. This is due to the construction of the $m\times m$ covariance matrix $C_{\widehat{\mathbf{Z}}'}$, and marginalisation of Gaussians in \eqref{eq:marginalisation-mean} and \eqref{eq:marginalisation-covariance} will take only $O\left(m^{2}\right)$ operations. Then, evaluation of the multivariate Gaussian CDF using the algorithms from \cite{Genz2002} are at most $O\left(m^{2}\right)$. We later illustrate that this formula yields a reasonable approximation.

\subsection{Success Probability Lower Bound}

Motivated by our approximation formula, we rigorously study Gaussian approximations of the form \eqref{eq:asymp_approx_distribution}, which allow us to analytically marginalise when approximating the success probability. The following lemma provides a sufficient condition for such an approximation to yield a lower bound on $p_{\mathrm{success}}$ in Problem \ref{prob:copula}\ref{prob:copula-part-b}. The result uses the following notion of stochastic dominance.
\begin{definition}[Multivariate stochastic dominance \cite{Shaked2007}]
We say that random vector $\mathbf{X}_{1} \in \mathbb{R}^{n}$ is stochastically dominated by random vector $\mathbf{X}_{2} \in \mathbb{R}^{n}$ and denote $\mathbf{X}_{1} \underset{\mathrm{st}}{\preceq} \mathbf{X}_{2}$ if and only if $\mathbb{E}\left[u\left(\mathbf{X}_{1}\right)\right] \leq \mathbb{E}\left[u\left(\mathbf{X}_{2}\right)\right]$ for all weakly increasing (i.e. non-decreasing) functions $u: \mathbb{R}^{n} \to \mathbb{R}$. Equivalently, we can say $\mathbf{X}_{1} \underset{\mathrm{st}}{\preceq} \mathbf{X}_{2}$ if and only if $\operatorname{Pr}\left(\mathbf{X}_{1} \in \mathbb{U}\right) \leq \operatorname{Pr}\left(\mathbf{X}_{2} \in \mathbb{U}\right)$ for all upper sets $\mathbb{U}$ (an upper set may be defined as a set which satisfies $\mathbf{x}_{2} \in \mathbb{U}$ for all $\mathbf{x}_{2} \geq \mathbf{x}_{1} \in \mathbb{U}$).
\label{def:stoch-dom}
\end{definition}
\begin{lemma}[Sufficient condition for lower bound]
Consider Problem \ref{prob:copula}\ref{prob:copula-part-b}. Let $\mathbf{Z}$ be a random vector for the first $m$ order statistics of $Z_{1}, \dots, Z_{n}$. Let $\widehat{\mathbf{Z}}$ be an arbitrary Gaussian approximation to $\mathbf{Z}$, of the form \eqref{eq:asymp_approx_distribution}. Suppose we have that $\mathbf{Z} \underset{\mathrm{st}}{\preceq} \widehat{\mathbf{Z}}$. Then the approximation computed in the same way as \eqref{eq:asymp_approx}, using $\widehat{\mathbf{Z}}'$ constructed symmetrically to $\widehat{\mathbf{Z}}$ as per \eqref{eq:symmetry_ord_stat}, yields a lower bound
\begin{equation}
\widehat{p}_{\mathrm{success}}\left(n, m, \alpha, \rho\right) \leq p_{\mathrm{success}}^{\mathrm{G}}\left(n, m, \alpha, \rho\right).
\end{equation}
\label{lem:lower-bound-uo}
\end{lemma}
\begin{proof}
Using the necessary and sufficient conditions for stochastic dominance of multivariate Gaussians \cite[Example 6.B.29]{Shaked2007}, and combined with Lemma \ref{lem:stoch_dom_param_rvects} applied to the conditional distribution \eqref{eq:gaussian-conditioning-rho}, since $\mathbf{Z} \underset{\mathrm{st}}{\preceq} \widehat{\mathbf{Z}}$ it follows that $\mathbf{X} \underset{\mathrm{st}}{\preceq} \widehat{\mathbf{X}}$. Thus
\begin{align}
\widehat{p}_{\mathrm{success}}\left(n, m, \alpha, \rho\right) &= \operatorname{Pr}\left(\min\left\{\widehat{\mathbf{X}}\right\} \leq x_{\alpha}^{*}\right) \\
&= 1 - \operatorname{Pr}\left(\widehat{\mathbf{X}} > x_{\alpha}^{*}\mathbf{1}\right) \\
&\leq  1 - \operatorname{Pr}\left(\mathbf{X} > x_{\alpha}^{*}\mathbf{1}\right) \\
&= \operatorname{Pr}\left(\min\left\{\mathbf{X}\right\} \leq x_{\alpha}^{*}\right) \\
&= p_{\mathrm{success}}^{\mathrm{G}}\left(n, m, \alpha, \rho\right).
\end{align}
\end{proof}
It is an open problem whether the proposed approximation formula in Section \ref{sec:gaussian_approximation} with \eqref{eq:approx-mean-vector} and \eqref{eq:approx-covariance-structure} satisfies the stochastic dominance condition of Lemma \ref{lem:lower-bound-uo}. This is owing to the CDF of $\mathbf{Z}$ taking on a complicated form (see \eqref{eq:joint_CDF_order_statistics}), while the CDF of $\widehat{\mathbf{Z}}$ is a multidimensional integral of a multivariate Gaussian density which has no analytical form. This makes it challenging to directly verify stochastic dominance. However, in the case of $m = 1$, it is possible to construct a Gaussian approximation for the first order statistic that is stochastically dominating. The following result presents a class of such approximations.
%One scheme to approximate order statistics using Gaussians is based on asymptotic theory. 

\begin{lemma}
Let $Z_{1:n}$ denote the first order statistic of an i.i.d. standard Gaussian sample of size $n$. For any $\theta \in \left(0, \frac{\pi}{2}\right)$, let
\begin{gather}
c_{1} = \dfrac{1}{2} - \dfrac{\theta}{\pi}, \\
c_{2} = \dfrac{\cot{\theta}}{\pi - 2\theta}.
\end{gather}
Consider $\widehat{Z}_{1:n} \sim \mathcal{N}\left(\mu_{n}, \sigma_{n}^{2}\right)$ where
\begin{gather}
\mu_{n} = -\sqrt{\dfrac{\log\left(nc_{1}\right)}{c_{2}}} \label{eq:mu-constructed} \\
\sigma_{n}^{2} = \dfrac{-\log\log 2}{2c_{2}\left(\log\left(nc_{1}\right) - \log\log 2\right)}. \label{eq:sigma-constructed}
\end{gather}
Then there exists some integer $n^{*}\left(\theta\right)$ such that for all $n \geq n^{*}\left(\theta\right)$, we have $Z_{1:n} \underset{\mathrm{st}}{\preceq} \widehat{Z}_{1:n}$.
\label{lem:stoch-dom-constructed}
\end{lemma}
\begin{proof}
See Appendix \ref{sec:proof-stoch-dom-constructed}.
\end{proof}

We also reveal the following property, which expresses monotonicity in the success probability with respect to $m$. This property intuitively says that having a larger selection size lends itself to `more chances' at success.

\begin{lemma}[Monotonicity with respect to selection size]
Given Problem \ref{prob:copula} and the triplet $\left(\bar{n}, \bar{\alpha}, \bar{\rho}\right) \in \mathbb{N} \times \left(0, 1\right) \times \left(0, 1\right)$, then
\begin{equation}
p_{\mathrm{success}}^{\mathrm{G}}\left(\bar{n}, m, \bar{\alpha}, \bar{\rho}\right) \leq p_{\mathrm{success}}^{\mathrm{G}}\left(\bar{n}, m', \bar{\alpha}, \bar{\rho}\right)
\end{equation}
for all $m \in \left[1, n\right)$ and $m' \in \left[m, n\right]$.
\label{lem:monotonicity-m}
\end{lemma}
\begin{proof}
The result follows from applying \eqref{eq:ord-opt-defn}, yielding
\begin{align}
p_{\mathrm{success}}^{\mathrm{G}}\left(\bar{n}, m, \bar{\alpha}, \bar{\rho}\right) &= \operatorname{Pr}\left(\bigcup_{i = 1}^{m}\left\{X_{\left\langle i\right\rangle} \leq x_{\alpha}^{*}\right\}\right) \\
&\leq \operatorname{Pr}\left(\bigcup_{i = 1}^{m'}\left\{X_{\left\langle i\right\rangle} \leq x_{\alpha}^{*}\right\}\right) \\
&= p_{\mathrm{success}}^{\mathrm{G}}\left(\bar{n}, m', \bar{\alpha}, \bar{\rho}\right),
\end{align}
where the inequality follows from $m \leq m'$.
\end{proof}

Combining Lemmas \ref{lem:lower-bound-uo}-\ref{lem:monotonicity-m}, we arrive at the following lower bound on the success probability under the Gaussian copula model.
\begin{theorem}[Success probability lower bound]
Consider Problem \ref{prob:copula}\ref{prob:copula-part-b}. Given $\theta \in \left(0, \frac{\pi}{2}\right)$, there exists some integer $n^{*}\left(\theta\right)$ such that for all $n \geq n^{*}\left(\theta\right)$, $m \in \left[1, n\right]$, $\rho \in \left(0, 1\right]$, $\alpha \in \left(0, 1\right]$
\begin{equation}
p_{\mathrm{success}}^{\mathrm{G}}\left(n, m, \alpha, \rho\right) \geq p_{\mathrm{success}}^{\mathrm{G}}\left(n, 1, \alpha, \rho\right) \geq \Phi\left(\dfrac{\Phi^{-1}\left(\alpha\right) - \rho\mu_{n}\left(\theta\right)}{\sqrt{1 - \rho^{2} + \rho^{2}\left[\sigma_{n}\left(\theta\right)\right]^{2}}}\right).
\label{eq:success-prob-copula-lb}
\end{equation}
Morever, given any $n \in \mathbb{N}$, $m \in \left[1, n\right]$, $\rho \in \left(0, 1\right]$, $\alpha \in \left(0, 1\right]$:
\begin{equation}
p_{\mathrm{success}}^{\mathrm{G}}\left(n, m, \alpha, \rho\right) \geq p_{\mathrm{success}}^{\mathrm{G}}\left(n, 1, \alpha, \rho\right) \geq \sup_{\theta \in \Theta_{n}}\Phi\left(\dfrac{\Phi^{-1}\left(\alpha\right) - \rho\mu_{n}\left(\theta\right)}{\sqrt{1 - \rho^{2} + \rho^{2}\left[\sigma_{n}\left(\theta\right)\right]^{2}}}\right)
\label{eq:success-prob-copula-lb-optimised}
\end{equation}
where $\Theta_{n} \subset \left(0, \frac{\pi}{2}\right)$ is the set of all $\theta$ such that $n \geq n^{*}\left(\theta\right)$, while $\mu_{n}\left(\theta\right)$, $\left[\sigma_{n}\left(\theta\right)\right]^{2}$ are \eqref{eq:mu-constructed}, \eqref{eq:sigma-constructed} respectively but with dependence on $\theta$ explicitly denoted.
\label{thm:constructed-lower-bound}
\end{theorem}
\begin{proof}
The right inequality in \eqref{eq:success-prob-copula-lb} occurs as a univariate special case of the approximation scheme \eqref{eq:asymp_approx}, in conjunction with the constructed stochastically dominating approximation in Lemma \ref{lem:stoch-dom-constructed} satisfying the conditions of Lemma \ref{lem:lower-bound-uo}. The left inequality in \eqref{eq:success-prob-copula-lb} is a result of monotonicity in $m$ from Lemma \ref{lem:monotonicity-m}. The inequalities in \eqref{eq:success-prob-copula-lb-optimised} follow naturally from the class of inequalities in \eqref{eq:success-prob-copula-lb}, by directly taking the supremum with respect to $\theta$.
\end{proof}
\begin{remark}
It is worthwhile to consider the smallest integer $n^{*}\left(\theta\right)$ such that \eqref{eq:success-prob-copula-lb} is valid. It is clear that we must have $n^{*}\left(\theta\right) > 1/c_{1}$, otherwise it possibly allows for $\log\left(nc_{1}\right) < 0$ in \eqref{eq:mu-constructed} and \eqref{eq:sigma-constructed}. Given $n$ and $\theta$, one can numerically confirm whether $n \geq n^{*}\left(\theta\right)$, using conditions from the proof of Lemma \ref{lem:stoch-dom-constructed}. We have empirically observed that $n^{*}\left(\theta\right)$ can be quite small, i.e. we can usually accept $n^{*}\left(\theta\right) = \left\lceil 1/c_{1}\right\rceil$. Using this numerical test, the optimised lower bound \eqref{eq:success-prob-copula-lb-optimised} can also be implemented via a numerical optimisation algorithm, noting that we need only conduct search over a univariate bounded region. Further discussion and pseudocode for these implementations can be found in Appendix \ref{sec:algorithms}.
\label{rem:numerical-check}
\end{remark}

We are now also equipped to state a characterisation on the monotonicity and convergence of $p_{\mathrm{success}}$ with respect to $n$. 

\begin{theorem}[Monotonicity and convergence with respect to sample size]
Consider Problem \ref{prob:copula}\ref{prob:copula-part-b}. Given the triplet $\left(\bar{m}, \bar{\alpha}, \bar{\rho}\right) \in \mathbb{N} \times \left(0, 1\right] \times \left(0, 1\right]$, then:
\begin{enumerate}[label=(\alph*)]
\item
for all $n \leq n'$ such that $n' \geq \bar{m}$ and $n \in \left[\bar{m}, n'\right]$, we have
\begin{equation}
p_{\mathrm{success}}^{\mathrm{G}}\left(n, \bar{m}, \bar{\alpha}, \bar{\rho}\right) \leq p_{\mathrm{success}}^{\mathrm{G}}\left(n', \bar{m}, \bar{\alpha}, \bar{\rho}\right),
\end{equation}\label{thm:approx-convergence-part-a}
\item moreover,
\begin{equation}
\lim_{n \to \infty}p_{\mathrm{success}}^{\mathrm{G}}\left(n, \bar{m}, \bar{\alpha}, \bar{\rho}\right) = 1.
\end{equation}\label{thm:approx-convergence-part-b}
\end{enumerate}
\label{thm:approx-convergence}
\end{theorem}
\begin{proof}
To prove part \ref{thm:approx-convergence-part-a}, it suffices to show that
\begin{equation}
p_{\mathrm{success}}^{\mathrm{G}}\left(n, m, \alpha, \rho\right) \leq p_{\mathrm{success}}^{\mathrm{G}}\left(n + 1, m, \alpha, \rho\right).
\end{equation}
Let $\mathbf{Z}$ denote the first $m$ order statistics of $Z$ with sample size $n$, and let $\widetilde{\mathbf{Z}}$ denote the first $m$ order statistics of $Z$ with sample size $n + 1$ instead. From Lemma \ref{lem:ord-stat-stoch-dom-increasing-n}, we have $\widetilde{\mathbf{Z}} \underset{\mathrm{st}}{\preceq} \mathbf{Z}$. Let $\mathbf{X} = \left(X_{\left\langle 1 \right\rangle}, \dots, X_{\left\langle m \right\rangle}\right)$ be corresponding $X$-values with sample size $n$, and let $\widetilde{\mathbf{X}}$ be the analogous random vector with sample size $n + 1$ instead. Using analogous arguments as in the proof of Lemma \ref{lem:lower-bound-uo}, it follows that $\widetilde{\mathbf{X}} \underset{\mathrm{st}}{\preceq} \mathbf{X}$. Thus
\begin{align}
p_{\mathrm{success}}^{\mathrm{G}}\left(n, m, \alpha, \rho\right) &= 1 - \operatorname{Pr}\left(\mathbf{X} > x_{\alpha}^{*}\mathbf{1}\right) \\
&\leq  1 - \operatorname{Pr}\left(\widetilde{\mathbf{X}} > x_{\alpha}^{*}\mathbf{1}\right) \\
&= p_{\mathrm{success}}^{\mathrm{G}}\left(n + 1, m, \alpha, \rho\right).
\end{align}
Part \ref{thm:approx-convergence-part-b} is evident from the lower bound in \eqref{eq:success-prob-copula-lb}, since
\begin{equation}
\lim_{n \to \infty}\Phi\left(\dfrac{\Phi^{-1}\left(\alpha\right) - \rho\mu_{n}}{\sqrt{1 - \rho^{2} + \rho^{2}\sigma_{n}^{2}}}\right) = 1.
\end{equation}
\end{proof}
In light of Theorem \ref{thm:approx-convergence}, we can also guarantee that we can attain an arbitrarily high probability of success, by letting $n$ become sufficiently large.
\begin{corollary}
Given the triplet $\left(\bar{m}, \bar{\alpha}, \bar{\rho}\right) \in \mathbb{N} \times \left(0, 1\right] \times \left(0, 1\right]$, and for any $\delta \in \left(0, 1\right]$, there exists an $\bar{n}\left(\bar{\alpha}, \bar{\rho}, \delta\right) < \infty$ such that
\begin{equation}
p_{\mathrm{success}}^{\mathrm{G}}\left(n, \bar{m}, \bar{\alpha}, \bar{\rho}\right) \geq 1 - \delta
\end{equation}
for all $n \geq \bar{n}$.
\label{coro:guarantee}
\end{corollary}

\subsection{Numerical Results for Approximation Formula}
\label{sec:numerical_results}

In Figures \ref{fig:approx_n}-\ref{fig:approx_xi}, we depict primarily the approximation formula from Section \ref{sec:gaussian_approximation}, when \eqref{eq:approx-mean-vector} and \eqref{eq:approx-covariance-structure} are used in \eqref{eq:asymp_approx}. This is compared against a Monte-Carlo estimate of the success probability (with $2 \times 10^{4}$ simulation replications used to generate each point), as well as the distribution-free bounds from Proposition \ref{prop:dist-free-bound}. In each figure, the nominal values $n = 100$, $m = 5$, $\alpha = 0.05$, $\rho = 0.6$ are used as a baseline, and one variable is varied at a time, keeping the others fixed. These figures validate the distribution-free bounds from Proposition \ref{prop:dist-free-bound}, but also show that they are not very useful in this example. The approximation formula is also exhibited to be reasonably close to the true probability, whilst still being conservative (in the sense of being an underestimate) of the success probability.

    \begin{figure}[hbt!]
       \centering
       \vskip\baselineskip
        \begin{subfigure}[t]{0.49\textwidth}
        \vskip 0pt
\includegraphics[width=\textwidth]{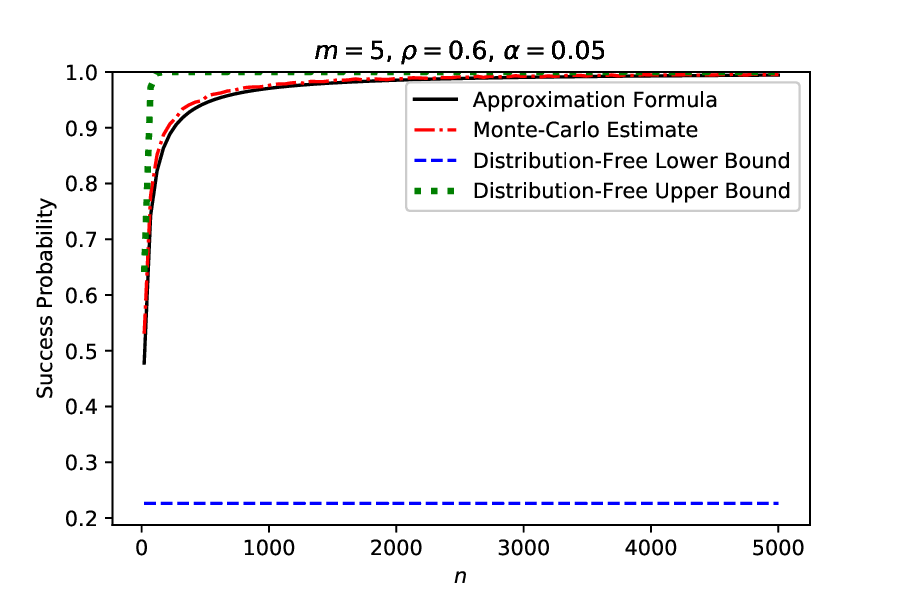}\centering
\caption{Comparison of $p_{\mathrm{success}}$ under the Gaussian copula model when $n$ is varied. In particular, this figure demonstrates monotonicity and convergence in $n$ from Theorem \ref{thm:approx-convergence}.}
\label{fig:approx_n}
        \end{subfigure}
        \hfill
        \begin{subfigure}[t]{0.49\textwidth}
        \vskip 0pt
\includegraphics[width=\textwidth]{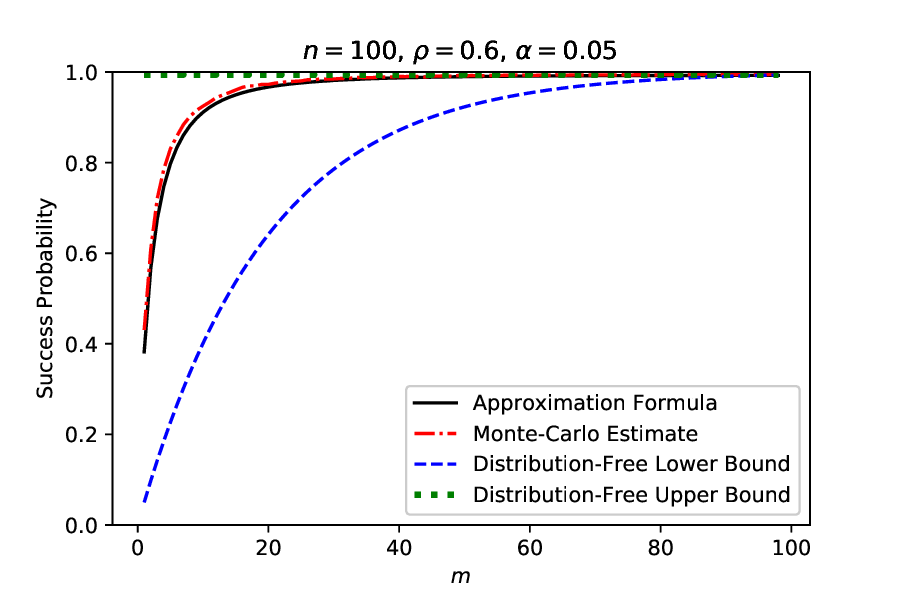}\centering
\caption{Comparison of $p_{\mathrm{success}}$ under the Gaussian copula model when $m$ is varied. In particular, this figure demonstrates: 1) monotonicity in $m$ from Lemma \ref{lem:monotonicity-m}, and 2) tightness of the bounds when $m = n$ from Remark \ref{rem:n-equal-m}.}
\label{fig:approx_m}
        \end{subfigure}
        \vskip\baselineskip
        \begin{subfigure}[t]{0.49\textwidth}   
\includegraphics[width=\textwidth]{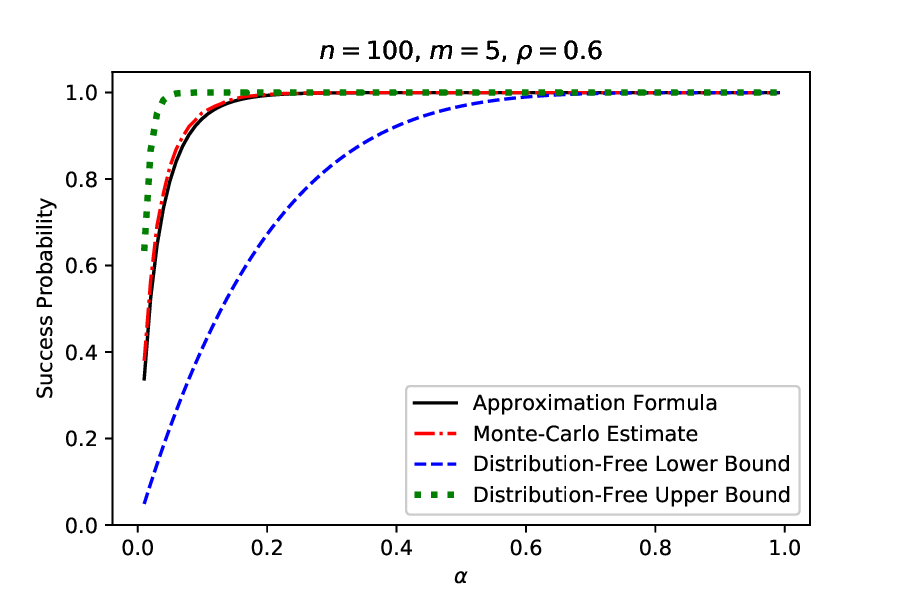}\centering
\caption{Comparison of $p_{\mathrm{success}}$ under the Gaussian copula model when $\alpha$ is varied. In particular, this figure demonstrates tightness of the bounds as $\alpha = 1$ and as $\alpha \to 0$ from Remark \ref{rem:n-equal-m}. The success probability appears to be increasing in $\alpha$, which is intuitive (by goal softening, we can improve our odds of success).}
\label{fig:approx_alpha}
        \end{subfigure}
        \hfill
        \begin{subfigure}[t]{0.49\textwidth}   
\includegraphics[width=\textwidth]{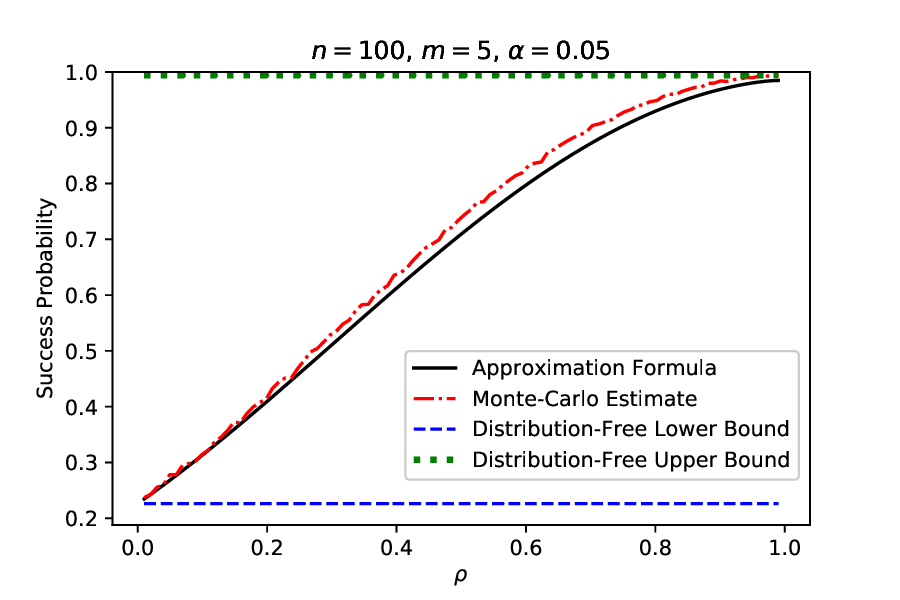}\centering
\caption{Comparison of $p_{\mathrm{success}}$ under the Gaussian copula model when $\rho$ is varied. The success probability appears to be increasing in $\rho$, or alternatively decreasing in $\xi$, which is intuitive (there is a higher price to be paid for more noise).} %In particular, this figure demonstrates: 1) monotonicity in $\rho$ from Corollary \ref{coro:monotonicity-rho}, and 2) convergence to the infinite-noise and noise-free bounds from Remark \ref{rem:xi}.}
\label{fig:approx_xi}
        \end{subfigure}
       \caption{\small Numerical results for the approximation formula.} 
        \label{fig:numerical1}
    \end{figure}
    
\subsection{Numerical Results for Lower Bound}

Throughout Figures \ref{fig:approx_n2_v3}-\ref{fig:approx_n_v3}, we plot the optimised lower bound \eqref{eq:success-prob-copula-lb-optimised} from Theorem \ref{thm:constructed-lower-bound}, with $m = 1$, $\alpha = 0.05$, $\rho = 0.5$ fixed, and $n$ varied. In Figure \ref{fig:approx_n2_v3}, this is compared against a Monte-Carlo estimate of the success probability (with $2 \times 10^{4}$ replications used to generate each point), and also the approximation formula from Section \ref{sec:gaussian_approximation}, when \eqref{eq:approx-mean-vector} and \eqref{eq:approx-covariance-structure} are used in \eqref{eq:asymp_approx}. In Figure \ref{fig:approx_n_v3}, we instead plot the lower bound over a semi-log horizontal axis scale for $n$, to illustrate the convergence of the success probability to one. These plots demonstrate that the lower bound is quite close to the approximation formula, which itself is reasonably close to the actual probability. As the lower bound has been derived with $m = 1$ while the bound itself does not change with $m$, this means the bound is least conservative for $m = 1$, and will generally become more conservative as $m$ grows. For instance with $\alpha = 0.05$ and $m = 20$, the lower bound from Proposition \ref{prop:dist-free-bound} yields $p_{\mathrm{success}} \geq 0.64$, surpassing the lower bounds from Figure \ref{fig:approx_n2_v3}.

 \begin{figure}[hbt!]
       \centering
       \vskip\baselineskip
        \begin{subfigure}[t]{0.49\textwidth}
        \vskip 0pt
\includegraphics[width=\textwidth]{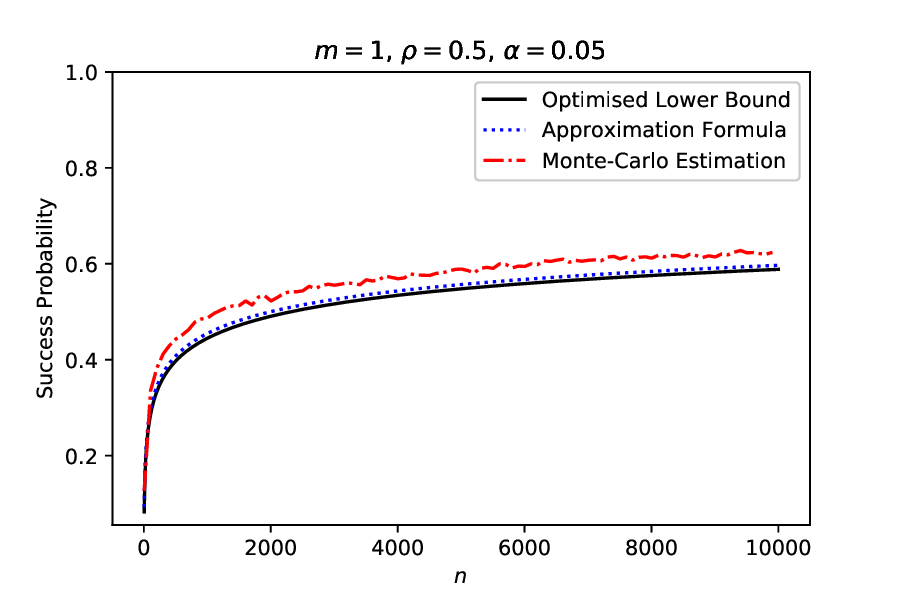}\centering
\caption{Comparison of the optimised lower bound in \eqref{eq:success-prob-copula-lb-optimised} of Theorem \ref{thm:constructed-lower-bound}, as $n$ is varied. In addition, this figure demonstrates monotonicity in $n$ from Theorem \ref{thm:approx-convergence}\ref{thm:approx-convergence-part-a}.}
\label{fig:approx_n2_v3}
        \end{subfigure}
        \hfill
        \begin{subfigure}[t]{0.49\textwidth}
        \vskip 0pt
\includegraphics[width=\textwidth]{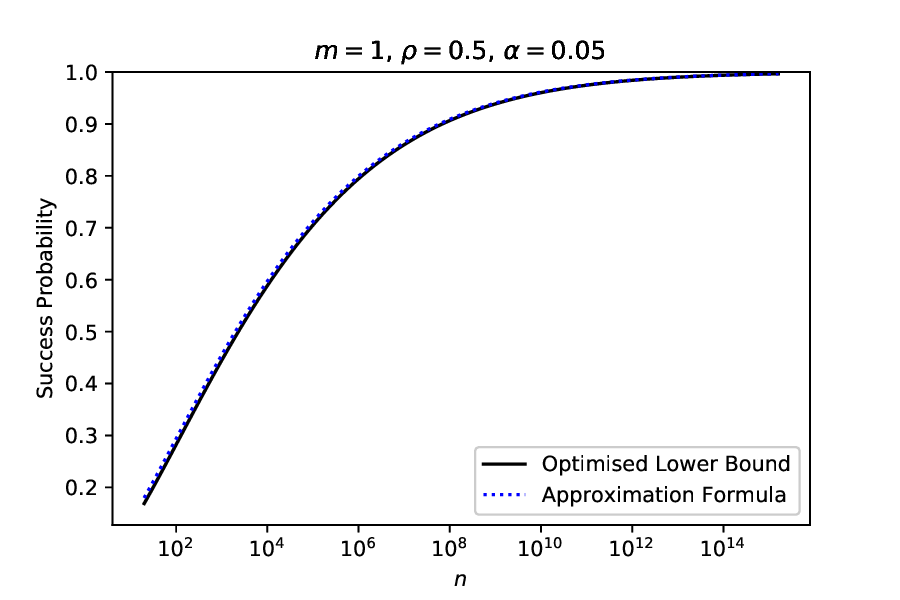}\centering
\caption{Comparison of the lower bound in Theorem \ref{thm:constructed-lower-bound} when $n$ is varied over a semi-log horizontal axis scale. In addition, this figure demonstrates convergence to one in $n$ from Theorem \ref{thm:approx-convergence}.}
\label{fig:approx_n_v3}
        \end{subfigure}
       \caption{\small Numerical results for the lower bound.} 
        \label{fig:numerical2}
    \end{figure}
    
\subsection{Computational Complexity Comparison}

%\begin{table}[htb!]
%\centering
%\caption{A table summarising the time complexity of various approaches to compute/estimate $p_{\mathrm{success}}$.}
%\begin{tabular}{@{}lll@{}}
%\toprule
%Computation Method     & Time Complexity       & Comments                                     \\ \midrule
%Monte-Carlo Simulation & $O\left(n\right)$     & Fixed simulation replications \\
%Numerical Quadrature   & $O\left(n\right)$     & Fixed error tolerance                   \\
%Approximation Formula  & $O\left(m^{2}\right)$ &                                              \\
%Optimised Lower Bound  & $O\left(1\right)$     & Fixed optimisation iterations \\ \bottomrule
%\end{tabular}
%\label{tab:complexity}
%\end{table}

Recall from Section \ref{sec:p-success} that evaluation of $p_{\mathrm{success}}$ via numerical quadrature integration is of time complexity $O\left(n\right)$. Estimation of $p_{\mathrm{success}}$ via Monte-Carlo simulation is unbiased (via the law of large numbers) and also of time complexity $O\left(n\right)$, for a fixed number of simulation replications. This is because each replication requires $O\left(n\right)$ time to find the lowest $m$ samples using the introselect algorithm \cite{Musser1997}, and subsequently $O\left(m\right)$ time (noting $m \leq n$) to determine if any of the lowest $m$ are smaller than $x_{\alpha}^{*}$. The corresponding standard deviation of the Monte-Carlo estimate is of $O\left(T^{-1/2}\right)$, with $T$ being the number of simulation replications. 

Recall from Section \ref{sec:gaussian_approximation} that the approximation formula is of time complexity $O\left(m^{2}\right)$. Thus, the approximation formula carries the advantage of being computationally more efficient for small $m$ and large $n$. Even for moderate $m$, the approximation method may still be the preferred approach due to the nested $O\left(1\right)$ computation via numerical integration, as well as the imprecision of the Monte-Carlo approach. Using the baseline values $n = 100$, $m = 5$, $\alpha = 0.05$, $\rho = 1/\sqrt{2}$ (equivalent to $\xi^{2} = 1$), we implemented the three different computation procedures in \textsc{Matlab} and compare the relative computation time. A Monte-Carlo simulation with $2\times 10^{8}$ replications yielded a $95\%$ confidence interval of $p_{\mathrm{success}} \in \left(0.90308, 0.90316\right)$. Using numerical quadrature integration, it took $0.9336$ of the time taken performing the Monte-Carlo simulation to compute $p_{\mathrm{success}} = 0.9031$. Using the approximation formula, it took $6\times 10^{-6}$ of the time taken for the Monte-Carlo simulation to compute $\widehat{p}_{\mathrm{success}} = 0.8765$.

In addition, implementation of the optimised lower bound \eqref{eq:success-prob-copula-lb-optimised} is of $O\left(1\right)$, since it is akin to the approximation formula with $m = 1$. The form of the lower bound from  \eqref{eq:success-prob-copula-lb} also admits an $O\left(1\right)$ approach to invert the lower bound for guaranteeing high probabilities of success; the advantage of this is exhibited in the following section.

\section{High Probability Guarantees of Success}

\label{sec:cases}

The purpose of this section is to numerically investigate the efficacy of Corollary \ref{coro:guarantee} in the large $n$ scenario, by considering the problem of guaranteeing a desired high probability.

\begin{problem}[High probability guarantees of success]
Given the triple $\left(m, \alpha, \rho\right) \in \mathbb{N} \times \left(0, 1\right] \times \left(0, 1\right]$, how large should $n$ be, such that
\begin{equation}
p_{\mathrm{success}}^{\mathrm{G}}\left(n, m, \alpha, \rho\right) \geq 1 - \delta
\end{equation}
for any given $\delta \in \left(0, 1\right]$?
\label{prob:high-prob}
\end{problem}

Having high probability guarantees would be useful in a situation where the sample size $n$ can be increased relatively cheaply (hence large $n$ would not be unreasonable), whereas increasing the selection size $m$ in order to increase $p_{\mathrm{success}}$ may be prohibitively expensive and/or non-negotiable. To address Problem \ref{prob:high-prob} (while avoiding the trivial solution of $n = \infty$), take the lower bound in \eqref{eq:success-prob-copula-lb} of Theorem \ref{thm:constructed-lower-bound}, which given $\alpha$, $\rho$ and $\delta$, we aim to invert for $n$ in terms of $\theta$ with the expression
\begin{equation}
\Phi\left(\dfrac{\Phi^{-1}\left(\alpha\right) - \rho\mu_{n}}{\sqrt{1 - \rho^{2} + \rho^{2}\sigma_{n}^{2}}}\right) = 1- \delta.
\end{equation}
Putting the definitions of $\mu_{n}$ and $\sigma_{n}$ from \eqref{eq:mu-constructed} and \eqref{eq:sigma-constructed} respectively, this equation can be rearranged into a quartic equation in $\sqrt{\log\left(nc_{1}\right)}$, of the form
\begin{equation}
a_{4}\log\left(nc_{1}\right)^{2} + a_{3}\log\left(nc_{1}\right)^{3/2} + a_{2}\log\left(nc_{1}\right) + a_{1}\log\left(nc_{1}\right)^{1/2} + a_{0} = 0,
\end{equation}
where
\begin{align}
a_{4} &= -\dfrac{2\rho^{2}}{\log\log 2}, \\
a_{3} &= -\dfrac{4\Phi^{-1}\left(\alpha\right)\rho\sqrt{c_{2}}}{\log\log2}, \\
a_{2} &= 2\rho^{2}-\dfrac{2c_{2}\left(\left[\Phi^{-1}\left(\alpha\right)\right]^{2}-\left[\Phi^{-1}\left(1-\delta\right)\right]^{2}+\rho^{2}\left[\Phi^{-1}\left(1-\delta\right)\right]^{2}\right)}{\log\log2}, \\
a_{1} &= 4\sqrt{c_{2}}\Phi^{-1}\left(\alpha\right)\rho, \\
a_{0} &= 2c_{2}\left(\left[\Phi^{-1}\left(\alpha\right)\right]^{2}-\left[\Phi^{-1}\left(1-\delta\right)\right]^{2}+\rho^{2}\left[\Phi^{-1}\left(1-\delta\right)\right]^{2}\right)-\rho^{2}\left[\Phi^{-1}\left(1-\delta\right)\right]^{2}.
\end{align}
Therefore we take the solution for $n$ corresponding to the greatest real root of the quartic equation. Let this solution for $n$ in terms of $\theta$ be denoted $\bar{n}\left(\theta\right)$. According to the monotonicity and convergence properties from Theorem \ref{thm:approx-convergence} manifested in Corollary \ref{coro:guarantee}, then provided $\bar{n}\left(\theta\right) \geq n^{*}\left(\theta\right)$, we guarantee
\begin{equation}
p_{\mathrm{success}}^{\mathrm{G}}\left(\bar{n}\left(\theta\right), m, \alpha, \rho\right) \geq 1 - \delta.
\end{equation}
since $\bar{n}\left(\theta\right)$ upper bounds the smallest $n$ needed such that $p_{\mathrm{success}} \geq 1 - \delta$. Moreover, one can numerically optimise with respect to $\theta$ to find the smallest $\bar{n}\left(\theta\right)$ that guarantees a high probability of success.

\begin{table}[hbt!]
\centering
\caption{Computed values of $n$ which guarantees $p_{\mathrm{success}}^{\mathrm{G}}\left(n, m, \alpha, \rho\right) \geq 1 - \delta$, with $\alpha = 0.01$ fixed and valid for any $m \geq 1$.}
\label{tab:high-prob}
\begin{tabular}{|l|c|c|c|}
\toprule
   & $\delta = 0 .01$        & $\delta = 0.05$         &  $\delta = 0.1$ \\ \midrule
$\rho = 0.01$ & $8.144\times 10^{47007}$       & $5.427\times 10^{34246}$ & $8.943 \times 10^{28267}$ \\ 
$\rho = 0.3$ & $3.289\times 10^{51}$       & $1.619\times 10^{38}$ & $8.775 \times 10^{31}$ \\ 
$\rho = 0.6$  & $8.703 \times 10^{11}$ & $1.988 \times 10^{9}$       & $1.078 \times 10^{8}$ \\ 
$\rho = 0.9$  & $16744$     & $4338$     & $2188$   \\ 
$\rho = 0.99$  & $893$     & $505$     & $372$   \\ \bottomrule
\end{tabular}
\end{table}

Table \ref{tab:high-prob} lists computed values of $\bar{n}$ numerically optimised with respect to $\theta$, taking into account the requirement $\bar{n}\left(\theta\right) \geq n^{*}\left(\theta\right)$, using the aforementioned approach. The table is valid for $m \geq 1$ (least conservative when $m = 1$), for $\alpha = 0.01$ and a variety of values for $\rho$ and $\delta$. The values for $n$ trend downwards as $\rho$ increases, which is intuitive (as fewer samples might be required if noisy observations are strongly correlated with the actual values). Of particular note, the case with extremely small correlation $\rho = 0.01$ (interpreted as a noise-to-signal ratio of $\xi^{2} \approx 10^{4}$, by the equivalence in Theorem \ref{thm:special_case}) requires $n$ to be at an impractical order of magnitude, namely $10^{47007}$ when $\delta = 0.01$. This highlights the utility of the lower bound in Theorem \ref{thm:constructed-lower-bound}, which allows for an $O\left(1\right)$ inversion to find a sufficiently high $n$. If Problem \ref{prob:high-prob} were attempted be solved with an $O\left(n\right)$ Monte-Carlo simulation or the $O\left(n\right)$ expression from Proposition \ref{prop:success-prob}, then large $n$ such as in the order of $10^{47007}$ would have rendered the evaluation of such probabilities to be intractable.

\section{Conclusion}
Supported by results applicable to an additive noise model, we have proposed various approaches for computing success probabilities under a general Gaussian copula model. Numerical experiments illustrate that our bounds and approximations reasonably follow the actual success probabilities. Furthermore, we used our lower bound to guarantee high probabilities of success, for situations which would have been computationally intractable by other means.

We pinpoint directions for further investigation. In practice, the correlation $\rho$ may not be known, but instead replaced by heuristic guesses. Also, a practitioner may prefer to not choose $\alpha$ directly, but rather let it be based on $x_{\alpha}^{*}$ corresponding to some concrete performance specification. Future work will focus on statistical uncertainty quantification techniques for obtaining reliable estimates of $\alpha$ and $\rho$ for use in computing success probabilities. Another research direction involves extending the stochastic dominance construction in Lemma \ref{lem:stoch-dom-constructed} to a multivariate version, or establishing that the approximation formula in Section \ref{sec:gaussian_approximation} satisfies the stochastic dominance condition in Lemma \ref{lem:lower-bound-uo}. Either may pave the way for improved lower bounds over Theorem \ref{thm:constructed-lower-bound}, for the case $m > 1$.

\bibliographystyle{ieeetr}
\bibliography{arxiv_template}  %%% Uncomment this line and comment out the ``thebibliography'' section below to use the external .bib file (using bibtex) .





% \end{thebibliography}

\begin{appendices}

\section{Proofs of Section \ref{sec:prelim} Results}
\label{sec:prelim-proofs}

\begin{lemma}[Joint CDF of order statistics]
Denote ranks $1 \leq n_{1} < \dots < n_{k} \leq n$. Then the joint CDF  of the order statistics $X_{n_{1}:n}, \dots, X_{n_{k}:n}$ for continuous $X$ with parent CDF denoted $F\left(\cdot\right)$ is
\begin{multline}
F_{n_{1},\dots,n_{k}}\left(x_{1}, \dots, x_{k}\right) = \sum_{i_{k} = n_{k}}^{n}\sum_{i_{k - 1} = n_{k - 1}}^{i_{k}}\dots\sum_{i_{1} = n_{1}}^{i_{2}}\dfrac{n!}{i_{1}!\left(i_{2} - i_{1}\right)!\times\dots\times\left(n - i_{k}\right)!}\left[F\left(x_{1}\right)\right]^{i_{1}} \\
\times\left[F\left(x_{2}\right) - F\left(x_{1}\right)\right]^{i_{2} - i_{1}}\times\dots\times\left[1 - F\left(x_{k}\right)\right]^{n - i_{k}}
\label{eq:joint_CDF_order_statistics}
\end{multline}
for the case $x_{1} \leq \dots \leq x_{k}$. For the case we do not have $x_{1} \leq \dots \leq x_{k}$, then it holds that
\begin{equation}
F_{n_{1},\dots,n_{k}}\left(x_{1}, \dots, x_{k}\right) = F_{n_{1},\dots,n_{k}}\left(x_{1}^{*}, \dots, x_{k}^{*}\right)
\label{eq:joint_CDF_order_statistics2}
\end{equation}
where
\begin{align}
x_{k}^{*} &:= x_{k}, \\
x_{k - 1}^{*} &:= \min\left\{x_{k - 1}, x_{k}^{*}\right\}, \\
&\vdots \nonumber \\
x_{1}^{*} &:= \min\left\{x_{1}, x_{2}^{*}\right\}.
\end{align}
\label{lem:ord-stat-joint-CDF}
\end{lemma}
\begin{proof}
The expression \eqref{eq:joint_CDF_order_statistics} generalises naturally based on arguments provided in \cite[Section 2.2]{David2005}. The result \eqref{eq:joint_CDF_order_statistics2} is obtained by noting that by construction $x_{1}^{*} \leq \dots \leq x_{k}^{*}$, and the joint density in the region bounded between $\left(x_{1}^{*}, \dots, x_{k}^{*}\right)$ and $\left(x_{1}, \dots, x_{k}\right)$ is zero.
\end{proof}

\subsection{Proof of Proposition \ref{prop:dist-free-bound}}
\label{sec:proof-dist-free-bound}

We prove the lower bound first. Let $G$ be the binomial random variable for the number of acceptable candidates of size $n$. Conditional on $G = g$ acceptable candidates, let $p_{\mathrm{success}|g}$ denote the conditional success probability. This can be interpreted as the the alignment probability (with alignment level one) from \cite[Eq. (2.19)]{Ho2007}, where it is shown that
\begin{equation}
p_{\mathrm{success}|g} \geq 1 - \binom{n - g}{m}\div \binom{n}{m}
\label{eq:conditional-lower-bound}
\end{equation}
The unconditional success probability may be obtained by using the law of total probability with the binomial distribution
\begin{align}
p_{\mathrm{success}} &\geq 1 - \sum_{g = 0}^{n}\binom{n}{g}\alpha^{g}\left(1 - \alpha\right)^{n - g}\binom{n - g}{m}\div \binom{n}{m} \\
&= 1 - \left(1 - \alpha\right)^{m}
\end{align}
after simplication. For the upper bound, we can write
\begin{equation}
p_{\mathrm{success}} = \operatorname{Pr}\left(\min\left\{X_{\left\langle 1\right\rangle}, \dots, X_{\left\langle m\right\rangle}\right\} \leq x_{\alpha}^{*}\right)
\end{equation}
and use the fact $\min_{i \in \left\{1, \dots, n\right\}}X_{i} \leq \min\left\{X_{\left\langle 1\right\rangle}, \dots, X_{\left\langle m\right\rangle}\right\}$ to bound
\begin{align}
p_{\mathrm{success}} &\leq \operatorname{Pr}\left(\min_{i \in \left\{1, \dots, n\right\}}\left\{X_{i}\right\} \leq x_{\alpha}^{*}\right) \\
&= 1 - \operatorname{Pr}\left(X_{1} > x_{\alpha}^{*}, \dots, X_{n} > x_{\alpha}^{*}\right) \\
&= 1 - \left(1 - \alpha\right)^{n}.
\end{align}

\subsection{Proof of Proposition \ref{prop:success-prob}}
\label{sec:proof-success-prob}

Let $G$ be the binomial random variable for the number of acceptable candidates of size $n$. We rely on the characterisation of the conditional failure probability given $G = g$, and denoted $p_{\mathrm{fail}|g}$, as being equal to the alignment probability in \cite[Eq. (2.19)]{Ho2007}. This says that
\begin{equation}
    p_{\mathrm{fail}|g} = \operatorname{Pr}\left(Z_{\left\{g + m\right\}} < Z_{\left\{1\right\}}\middle| G = g\right),
\end{equation}
where $Z_{\left\{g + m\right\}}$ and $Z_{\left\{1\right\}}$ are to be explained as follows. The random variable $Z_{\left\{g + m\right\}}$ is the $m$\textsuperscript{th} order statistic from a sample of size $n - g$, with parent distribution being the distribution of
\begin{equation}
    \overline{Z} = Y + \overline{X},
\end{equation}
where $\overline{X}$ is the left-truncated version of $X$, truncated at $x_{\alpha}^{*}$. The PDF of $\overline{X}$ is $f_{\overline{X}}\left(x\right)$ and given in \eqref{eq:left-truncated-pdf}. The PDF of $\overline{Z}$ can be obtained via convolution \cite[Section 4.8]{Grimmett2001}; its CDF is $F_{\overline{Z}}\left(z\right)$ and given in \eqref{eq:sum-dist-upper-cdf}. Then the CDF of $Z_{\left\{g + m\right\}}$ is $F_{Z_{\left\{g+m\right\}}}\left(z\right)$ and given in \eqref{eq:upper-conditional-cdf}, from which the expression can be obtained as a special case of Lemma \ref{lem:ord-stat-joint-CDF}. Similarly, $Z_{\left\{1\right\}}$ is the first order statistic of a sample of size $g$, with parent distribution being the distribution of
\begin{equation}
    \underline{Z} = Y + \underline{X},
\end{equation}
where $\underline{X}$ is the right-truncated version of $X$, truncated at $x_{\alpha}^{*}$. The PDF of $\underline{X}$ is $f_{\underline{X}}\left(x\right)$ and given in \eqref{eq:right-truncated-pdf}. The PDF of $\underline{Z}$ is $f_{\underline{Z}}\left(z\right)$ and can be obtained via convolution, given in \eqref{eq:sum-dist-lower-pdf}. The CDF of $\underline{Z}$ is $F_{\underline{Z}}\left(z\right)$ and given in \eqref{eq:sum-dist-lower-cdf}. Then the PDF of $Z_{\left\{1\right\}}$ is given by \eqref{eq:lower-conditional-pdf}, which can be obtained by differentiating a special case of the CDF in Lemma \ref{lem:ord-stat-joint-CDF}, or by directly using well-known expressions for densities of order statistics \cite[Eq. (2.1.6)]{David2005}. Then from the characterisation of the conditional failure probability $p_{\mathrm{fail}|g}$, along with the law of total probability, we get
\begin{align}
p_{\mathrm{success}} &= 1 - \sum_{g = 0}^{n}\binom{n}{g}\alpha^{g}\left(1 - \alpha\right)^{n - g}p_{\mathrm{fail}|g} \\
&= 1 - \left(1 - \alpha\right)^{n} - \sum_{g = 1}^{n - m}\binom{n}{g}\alpha^{g}\left(1 - \alpha\right)^{n - g}\operatorname{Pr}\left(Z_{\left\{g + m\right\}} < Z_{\left\{1\right\}}\middle| G = g\right) \label{eq:law_total_prob_conditional_failure}
\end{align}
because $p_{\mathrm{fail}|0} = 1$ and $p_{\mathrm{fail}|g} = 0$ for $g > n - m$. Then \eqref{eq:p-success-expresion} follows from \eqref{eq:law_total_prob_conditional_failure} because $Z_{\left\{g + m\right\}}$ and $Z_{\left\{1\right\}}$ are conditionally independent given $G$.

\section{Proofs of Section \ref{sec:main} Results}
\label{sec:main-proofs}

\begin{lemma}[Stochastic dominance of parametrised random vectors]
\label{lem:stoch_dom_param_rvects}
For random vectors $\mathbf{X}$, $\boldsymbol{\Theta}$, consider the conditional distribution $\left[\mathbf{X}\middle|\boldsymbol{\Theta}\right]$. Suppose that $\left[\mathbf{X}\middle|\boldsymbol{\Theta} = \boldsymbol{\theta}_{1}\right] \underset{\mathrm{st}}{\preceq} \left[\mathbf{X}\middle|\boldsymbol{\Theta} = \boldsymbol{\theta}_{2}\right]$ whenever $\boldsymbol{\theta}_{1} \leq \boldsymbol{\theta}_{2}$. Let $\mathbf{X}_{1}$ denote the variable for $\mathbf{X}$ that arises from chaining the random vector $\boldsymbol{\Theta}_{1}$ with $\left[\mathbf{X}\middle|\boldsymbol{\Theta}_{1}\right]$, and similarly let $\mathbf{X}_{2}$ denote the variable for $\mathbf{X}$ that arises from chaining the random vector $\boldsymbol{\Theta}_{2}$ with $\left[\mathbf{X}\middle|\boldsymbol{\Theta}_{2}\right]$. If $\boldsymbol{\Theta}_{1} \underset{\mathrm{st}}{\preceq} \boldsymbol{\Theta}_{2}$, then
\begin{equation}
\mathbf{X}_{1} \underset{\mathrm{st}}{\preceq} \mathbf{X}_{2}.
\end{equation}
\end{lemma}
\begin{proof}
For all upper sets $\mathbb{U}$, we may write using indicator variables (denoted by $\mathbb{I}$)
\begin{align}
\operatorname{Pr}\left(\mathbf{X}_{1} \in \mathbb{U}\right) &= \mathbb{E}\left[\mathbb{I}_{\left\{\mathbf{X}_{1} \in \mathbb{U}\right\}}\right] \\
&= \mathbb{E}_{\boldsymbol{\Theta}_{1}}\left[\mathbb{E}\left[\mathbb{I}_{\left\{\mathbf{X} \in \mathbb{U}\right\}}\middle|\boldsymbol{\Theta}_{1}\right]\right] \\
&= \mathbb{E}_{\boldsymbol{\Theta}_{1}}\left[\operatorname{Pr}\left(\mathbf{X} \in \mathbb{U}\middle|\boldsymbol{\Theta}_{1}\right)\right] \\
&\leq \mathbb{E}_{\boldsymbol{\Theta}_{2}}\left[\operatorname{Pr}\left(\mathbf{X} \in \mathbb{U}\middle|\boldsymbol{\Theta}_{2}\right)\right] \\
&= \operatorname{Pr}\left(\mathbf{X}_{2} \in \mathbb{U}\right)
\end{align}
where the inequality follows by definition of stochastic dominance, since $\operatorname{Pr}\left(\mathbf{X} \in \mathbb{U}\middle|\boldsymbol{\theta}\right)$ is a weakly increasing function of $\boldsymbol{\theta}$ for all upper sets $\mathbb{U}$.
\end{proof}

\begin{lemma}
Let $\mathbf{Z}_{\left[m\right]:n} := \left(Z_{1:n}, \dots, Z_{m:n}\right)$ denote the joint first $m$ order statistics of an i.i.d. sample of size $n$ from parent distribution $Z$. Then
\begin{equation}
\mathbf{Z}_{\left[m\right]:\left(n + 1\right)} \underset{\mathrm{st}}{\preceq} \mathbf{Z}_{\left[m\right]:n}.
\label{eq:stoch-dom-nplus1}
\end{equation}
\label{lem:ord-stat-stoch-dom-increasing-n}
\end{lemma}
\begin{proof}
Consider the following construction on the same probability space. Form an i.i.d. sample of size $n + 1$ from $Z$ and take the first $m$ order statistics. This will be equal in law to $\mathbf{Z}_{\left[m\right]:\left(n + 1\right)}$. Now delete one element uniformly at random, and re-compute the first $m$ order statistics. This will be equal in law to $\mathbf{Z}_{\left[m\right]:n}$. Moreover, for every realisation (denoted in lowercase) from this probability space, we have
\begin{equation}
\left(z_{1:\left(n + 1\right)}, \dots, z_{m:\left(n + 1\right)}\right) \leq \left(z_{1:n}, \dots, z_{m:n}\right).
\end{equation}
Therefore from the characterisation of stochastic dominance in \cite[Theorem 6.B.1]{Shaked2007}, \eqref{eq:stoch-dom-nplus1} holds.
\end{proof}

\begin{lemma}
For all $p \in \left[0, \frac{1}{2}\right]$.
\begin{equation}
-p\log 4 \leq \log\left(1 - p\right) \leq -p.
\end{equation}
\label{lem:log-concave-bounds}
\end{lemma}
\begin{proof}
The lower bound can be established over $p \in \left[0, \frac{1}{2}\right]$ via concavity of $\log\left(1 - p\right)$, i.e. line secants lie below the graph. The upper bound can also be established over $p \geq 0$ via concavity.
\end{proof}

\begin{lemma}
For the Gaussian $Q$-function given by $Q\left(x\right) = 1 - \Phi\left(x\right)$
\begin{equation}
c_{1}\exp\left(-c_{2}x^{2}\right) \leq Q\left(x\right) \leq \frac{1}{2}\exp\left(-\dfrac{x^{2}}{2}\right)
\end{equation}
over $x \geq 0$, with
\begin{gather}
c_{1} = \dfrac{1}{2} - \dfrac{\theta}{\pi} \\
c_{2} = \dfrac{\cot{\theta}}{\pi - 2\theta}
\end{gather}
for any $\theta \in \left(0, \frac{\pi}{2}\right)$.
\label{lem:q-function-bounds}
\end{lemma}
\begin{proof}
The lower bound is due to \cite[Eq. (2)]{Wu2018} and the upper bound is found in \cite[Eq. (5)]{Chiani2003}.
\end{proof}

\subsection{Proof of Lemma \ref{lem:stoch-dom-constructed}}
\label{sec:proof-stoch-dom-constructed}

If $\widehat{Z}_{1:n}$ stochastically dominates $Z_{1:n}$, then $\operatorname{Pr}\left(\widehat{Z}_{1:n} \geq z\right) \geq \operatorname{Pr}\left(Z_{1:n} \geq z\right)$ for all $z \in \mathbb{R}$. Or in terms of the Gaussian $Q$-function $Q\left(z\right) = 1 - \Phi\left(z\right)$, we require
\begin{equation}
Q\left(z\right)^{n} \leq Q\left(\dfrac{z - \mu_{n}}{\sigma_{n}}\right)
\label{eq:q-function-stoch-dom}
\end{equation}
for all $z \in \mathbb{R}$, where the left-hand side can be derived with a special case of Lemma \ref{lem:ord-stat-joint-CDF}. The idea is to show that this bound holds over three different intervals whose union is $\mathbb{R}$, being $\left(-\infty, \mu_{n}\right]$, $\left[\mu_{n}, 0\right]$ and $\left[0, \infty\right)$. We begin with $z \in \left(-\infty, \mu_{n}\right]$. Since $\mu_{n} \leq 0$, then via the lower bound in Lemma \ref{lem:q-function-bounds}
\begin{equation}
Q\left(z\right)^{n} \leq \left(1 - c_{1}\exp\left(-c_{2}z^{2}\right)\right)^{n}.
\end{equation}
Since $0 \leq c_{1}\exp\left(-c_{2}z^{2}\right) \leq 1/2$, then putting $p = c_{1}\exp\left(-c_{2}z^{2}\right)$ in the upper bound from Lemma \ref{lem:log-concave-bounds}, we get
\begin{align}
Q\left(z\right)^{n} &\leq \exp\left(-nc_{1}\exp\left(-c_{2}z^{2}\right)\right) \\
&= \exp\left(-\exp\left(-\left(c_{2}z^{2}-\log\left(nc_{1}\right)\right)\right)\right).
\end{align}
Now using the upper bound in Lemma \ref{lem:q-function-bounds}, we have for $z \leq \mu_{n}$:
\begin{equation}
1-\dfrac{1}{2}\exp\left(-\dfrac{\left(z-\mu_{n}\right)^{2}}{2\sigma_{n}^{2}}\right)\leq Q\left(\dfrac{z-\mu_{n}}{\sigma_{n}}\right).
\end{equation}
The lower bound in Lemma \ref{lem:log-concave-bounds} implies $\exp\left(-p\log 4\right) \leq 1 - p$. Applying this with $p = \frac{1}{2}\exp\left(-\frac{\left(z-\mu_{n}\right)^{2}}{2\sigma_{n}^{2}}\right)$ and after some manipulation, we arrive at
\begin{equation}
Q\left(\dfrac{z-\mu_{n}}{\sigma_{n}}\right) \geq \exp\left(-\exp\left(-\left(\dfrac{\left(z-\mu_{n}\right)^{2}}{2\sigma_{n}^{2}}-\log\log2\right)\right)\right).
\end{equation}
Thus a sufficient condition for $Q\left(z\right)^{n} \leq Q\left(\frac{z - \mu_{n}}{\sigma_{n}}\right)$ over $z \in \left(-\infty, \mu_{n}\right]$ is
\begin{equation}
\exp\left(-\exp\left(-\left(c_{2}z^{2}-\log\left(nc_{1}\right)\right)\right)\right) \leq \exp\left(-\exp\left(-\left(\dfrac{\left(z-\mu_{n}\right)^{2}}{2\sigma_{n}^{2}}-\log\log2\right)\right)\right)
\end{equation}
or equivalently,
\begin{equation}
\left(\dfrac{1}{\sigma_{n}^{2}}-2c_{2}\right)z^{2}-2\dfrac{\mu_{n}}{\sigma_{n}^{2}}z+\dfrac{\mu_{n}^{2}}{\sigma_{n}^{2}}+2\log\left(nc_{1}\right)-2\log\log2\geq0.
\end{equation}
The roots of this quadratic are at
\begin{equation}
z = \dfrac{\mu_{n}\pm\sqrt{\mu_{n}^{2}-\left(1-2c_{2}\sigma_{n}^{2}\right)\left(\mu_{n}^{2}+2\sigma_{n}^{2}\log\left(nc_{1}\right)-2\sigma_{n}^{2}\log\log2\right)}}{1-2c_{2}\sigma_{n}^{2}}
\end{equation}
with discriminant $\Delta$ calculated by
\begin{equation}
\Delta = \sigma_{n}^{4}\left(4c_{2}\log\left(nc_{1}\right)-4c_{2}\log\log2\right)+\sigma_{n}^{2}\left(2c_{2}\mu_{n}^{2}-2\log\left(nc_{1}\right)+2\log\log2\right).
\end{equation}
Under the same choice of $\theta$, note $2c_{2}\mu_{n}^{2}=2\log\left(nc_{1}\right)$ and the discriminant becomes
\begin{equation}
\Delta=\sigma_{n}^{4}\left(4c_{2}\log\left(nc_{1}\right)-4c_{2}\log\log2\right)+\sigma_{n}^{2}\left(2\log\log2\right).
\end{equation}
The quadratic inequality is satisfied everywhere if the discriminant is non-positive, so put $\Delta = 0$ and taking the positive solution for $\sigma_{n}^{2}$, giving
\begin{equation}
\sigma_{n}^{2} = \dfrac{-\log\log 2}{2c_{2}\left(\log\left(nc_{1}\right) - \log\log 2\right)}.
\end{equation}
Therefore the inequality is satisfied provided $nc_{1} > 1$, which occurs for sufficiently large $n$, since $c_{1} > 0$. Next we show that the stochastic dominance condition is satisfied for $z \in \left[\mu_{n}, 0\right]$, under the proposed choice of $\mu_{n}$ and $\sigma_{n}$ above. Over this interval, we can use the same upper bound on $Q\left(z\right)^{n}$ as before, and now we have the lower bound
\begin{equation}
Q\left(\dfrac{z-\mu_{n}}{\sigma_{n}}\right) \geq \exp\left(-\left(c_{2}\left(\dfrac{z-\mu_{n}}{\sigma_{n}}\right)^{2}-\log c_{1}\right)\right).
\end{equation}
Thus we want to show that
\begin{equation}
nc_{1}\exp\left(-c_{2}z^{2}\right)\geq c_{2}\left(\dfrac{z-\mu_{n}}{\sigma_{n}}\right)^{2}-\log c_{1}.
\end{equation}
Fix $z$, and recognise that $\frac{\mu_{n}^{2}}{\sigma_{n}^{2}} = O\left(\left(\log n\right)^{2}\right)$ in the right-hand side, while the left-hand side is $O\left(n\right)$. Therefore
\begin{equation}
O\left(n\right) \geq \left(\left(\log n\right)^{2}\right)
\end{equation}
since $O\left(e^{n}\right) \geq O\left(n^{2}\right)$. Lastly for the interval $z \in \left[0, \infty\right)$, we use the upper bound in Lemma \ref{lem:q-function-bounds} to give
\begin{equation}
Q\left(z\right)^{n} \leq \dfrac{1}{2^{n}}\exp\left(-\dfrac{nz^{2}}{2}\right)
\end{equation}
and we can use the same lower bound as in the preceding interval. In the same vein as above, we want to show
\begin{equation}
\left(\dfrac{n}{2}-\dfrac{c_{2}}{\sigma_{n}^{2}}\right)z^{2}-\dfrac{2c_{2}\mu_{n}}{\sigma_{n}^{2}}+n\log2+\dfrac{\mu_{n}^{2}}{\sigma_{n}^{2}}-\log c_{1}\geq0.
\end{equation}
The discriminant of the quadratic is non-positive when
\begin{equation}
\left(\dfrac{n}{2}-\dfrac{c_{2}}{\sigma_{n}^{2}}\right)\left(n\log2+\dfrac{\mu_{n}^{2}}{\sigma_{n}^{2}}-\log c_{1}\right)\geq\dfrac{c_{2}^{2}\mu_{n}^{2}}{\sigma_{n}^{4}}.
\label{eq:right-interval-stoch-dom}
\end{equation}
The left-hand side is $O\left(n^{2}\right)$ and the right-hand side is $O\left(\left(\log n\right)^{3}\right)$, thus this inequality is also satisfied for sufficiently large $n$.

\section{Implementation of Lower Bounds in Theorem \ref{thm:constructed-lower-bound}}
\label{sec:algorithms}

In continuation of the discussion from Remark \ref{rem:numerical-check}, the lower bounds in Theorem \ref{thm:constructed-lower-bound} can be implemented numerically. This is done by using sufficient conditions found in the proof of Lemma \ref{lem:stoch-dom-constructed} to check whether $n \geq n^{*}\left(\theta\right)$ for a given $n$ and $\theta$. We are required to check whether the inequality \eqref{eq:q-function-stoch-dom} is satisfied over each of the intervals $\left(\infty, \mu_{n}\right]$, $\left[\mu_{n}, 0\right]$ and $\left[0, \infty\right)$. The inequality is satisfied over $\left(\infty, \mu_{n}\right]$ by construction provided $nc_{1} > 1$, whereas \eqref{eq:right-interval-stoch-dom} contains the sufficient condition for the interval $\left[0, \infty\right)$. As for the bounded interval $\left[\mu_{n}, 0\right]$, we can directly evaluate (up to the available numerical precision) whether \eqref{eq:q-function-stoch-dom} is satisfied. Pseudocode to implement this numerical test is provided in Algorithm \ref{alg:test}. Using this test, we can implement the optimised lower bound \eqref{eq:success-prob-copula-lb-optimised} in Theorem \ref{thm:constructed-lower-bound}. Pseudocode for this is found in Algorithm \ref{alg:lb}.

\begin{algorithm}[H]
\caption{Numerical test of sufficient conditions for $n \geq n^{*}\left(\theta\right)$ in Theorem \ref{thm:constructed-lower-bound}}
\label{alg:test}
\begin{algorithmic}[1]
\Function{NumericalTest}{$n$, $\theta$}
\State $c_{1} \gets \dfrac{1}{2} - \dfrac{\theta}{\pi}, \quad c_{2} \gets \dfrac{\cot{\theta}}{\pi - 2\theta}$
\State $\mu_{n} \gets -\sqrt{\dfrac{\log\left(nc_{1}\right)}{c_{2}}}, \quad \sigma_{n}^{2} \gets \dfrac{-\log\log 2}{2c_{2}\left(\log\left(nc_{1}\right) - \log\log 2\right)}$
\If{$nc_{1} \leq 1$}  \Comment Check sufficient condition for the interval $\left(\infty, \mu_{n}\right]$
    \State \Return $\mathtt{False}$
\ElsIf{\eqref{eq:q-function-stoch-dom} fails over $\left[\mu_{n}, 0\right]$} \Comment Check sufficient condition for the interval $\left[\mu_{n}, 0\right]$
    \State \Return $\mathtt{False}$
\ElsIf{\eqref{eq:right-interval-stoch-dom} fails} \Comment Check sufficient condition for the interval $\left[0, \infty\right)$
    \State \Return $\mathtt{False}$
\Else
    \State \Return $\mathtt{True}$
\EndIf
\EndFunction
\end{algorithmic}
\end{algorithm}

\begin{algorithm}[H]
\caption{Implementation of lower bounds in Theorem \ref{thm:constructed-lower-bound}}
\label{alg:lb}
\begin{algorithmic}[1]
\Function{LowerBound}{$n$, $\alpha$, $\rho$, $\theta$} \Comment{Lower bound in \eqref{eq:success-prob-copula-lb}}
\If{\Call{\textsc{NumericalTest}}{$n$, $\theta$}}
\State \Return Right-hand side of \eqref{eq:success-prob-copula-lb}
\Else
\State \Return $0$
\EndIf
\EndFunction
\Function{OptimisedLowerBound}{$n$, $\alpha$, $\rho$} \Comment{Optimised lower bound in \eqref{eq:success-prob-copula-lb-optimised}}
\State \Return $\max_{\theta \in \left(0, \pi/2\right)}$\Call{\textsc{LowerBound}}{$n$, $\alpha$, $\rho$, $\theta$}
\EndFunction
\end{algorithmic}
\end{algorithm}

\end{appendices}

\end{document}